\documentclass[reqno]{amsart}\usepackage{amsmath}

\usepackage{amsfonts}

\newtheorem{theorem}{Theorem}
\theoremstyle{plain}
\newtheorem{acknowledgement}{Acknowledgement}

\newtheorem{corollary}{Corollary}

\newtheorem{definition}{Definition}

\newtheorem{lemma}{Lemma}

\newtheorem{proposition}{Proposition}
\newtheorem{remark}{Remark}

\DeclareMathOperator{\Div}{div}
 \numberwithin{equation}{section}
 \numberwithin{theorem}{section}
 \numberwithin{proposition}{section}
 \numberwithin{remark}{section}
 \numberwithin{definition}{section}
 \numberwithin{lemma}{section}
 \numberwithin{corollary}{section}
 \numberwithin{example}{section}
 \numberwithin{claim}{section}
\input{tcilatex}

\begin{document}
\title[Acoustics of almost periodic porous media ]{Asymptotic behaviour of
viscoelastic composites with almost periodic microstructures}
\author{Jean Louis Woukeng}
\address{Department of Mathematics and Computer Science, University of
Dschang, P.O. Box 67, Dschang, Cameroon}
\email{jwoukeng@yahoo.fr}
\date{}

\begin{abstract}
In this paper we study the acoustic properties of porous media saturated by
an incompressible viscoelastic fluid. The model considered here consists of
a linear deformable porous skeleton having memory that is surrounded by a
viscoelastic Oldroyd fluid. Assuming the microstructures to be almost
periodically distributed and under the almost periodicity hypothesis on the
coefficients of the governing equations, we determine the macroscopic
equivalent medium. To achieve our goal, we use some very recent tools about
the sigma convergence of convolution sequences.
\end{abstract}

\maketitle

\section{Introduction and the model}

A number of important engineering materials such as polymers store and
dissipate simultaneously mechanical energy when subjected to external
forces. This form of response as a combination of both liquid-like and
solid-like features is termed viscoelasticity. The constitutive equations
for the viscoelastic behaviour include the elastic deformation and viscous
flow as special cases, and at the same time provide responses which are
characteristic of the combined behaviour. In order to deal with viscoelastic
fluid motion, researchers have focused on the use of constitutive equations
that capture important qualitative features describing the relationships
between the kinematic, mechanical and thermal fields equations and hence
permit the formulation of well-posed problems that are simplistic relative
the formulations that are supposed to characterize real materials. In that
direction, many models have been proposed and developed. We refer e.g. to 
\cite{Agra, AAGM, BA, Dasser, DGHL1, GM, Joseph, Pani, Sobo, Wang}, just to
cite a few. It is worth noting that in all the previous works, the models
involved did not deal with skeletons with memory.

In this work, we mean to study the acoustic properties of porous saturated
media. In this respect we consider the problem of viscoelastic Oldroyd fluid
motion through linear deformable porous media with memory. The novelty is
threefold. First, the deformable skeleton has memory. Second, instead of
considering usual Stokes equations flow, we rather consider the Oldroyd
class of incompressible viscoelastic fluids. Third, the pores are
distributed in an almost periodic fashion, generalizing the usual
equidistribution (also known as periodic distribution) always considered in
the earlier works. To be precise, as stated before, the problem will be
discussed in terms of two particular types of ideal materials: the elastic
skeleton (with memory) and the viscoelastic fluid, and the constitutive
equations will therefore be described on the microscale as follows.

\subsection{Description of the domain}

Let $Y=(0,1)^{N}$ ($N=2$ or $3$ being fixed once for all here and in the
sequel) be the reference cell, and let $Y_{1}$ and $Y_{2}$ be two open
disjoint subsets of $Y$ representing the local structure of the skeleton
(which is a linear elastic material) and the local structure of the fluid,
respectively. We assume that $\overline{Y}_{1}\subset Y$, $Y=\overline{Y}%
_{1}\cup Y_{2}$, $Y_{2}$ is connected and that the boundary $\partial Y_{1}$
of $Y_{1}$ is Lipschitz continuous. Moreover we assume that $Y_{1}$ and $%
Y_{2}$ have positive Lebesgue measure. Let $S\subset \mathbb{Z}^{N}$ be an
infinite subset of $\mathbb{Z}^{N}$ and let 
\begin{equation*}
G_{1}=\cup _{k\in S}(k+Y_{1}),\ \ \ \ G_{2}=\mathbb{R}^{N}\backslash 
\overline{G}_{1}.
\end{equation*}%
Then $G_{j}$ ($j=1,2$) is an open subset of $\mathbb{R}^{N}$ and moreover $%
G_{2}$ is connected. Since the cells $(k+Y)_{k\in S}$ are pairwise disjoint,
the characteristic function $\chi _{1}$ of the set $G_{1}$ in $\mathbb{R}%
^{N} $ satisfies $\chi _{1}=\sum_{k\in S}\chi _{k+Y_{1}}$ or more precisely, 
\begin{equation*}
\chi _{1}=\sum_{k\in \mathbb{Z}^{N}}\theta (k)\chi _{k+Y_{1}}
\end{equation*}%
where $\theta $ is the characteristic function of $S$ in $\mathbb{Z}^{N}$.
The function $\theta $ is the \emph{distribution function of the pores}
(that is, it characterizes the porous cells; remind that not all the cells
are porous). The special case $\theta (k)=1$ for all $k\in \mathbb{Z}^{N}$
(i.e. $S=\mathbb{Z}^{N}$) leads to the equidistribution of fissured cells,
commonly known in the literature as the \emph{periodic porous medium}. Here
we assume the function $\theta $ to be almost periodic (which corresponds to
the almost periodic distribution of the pores, that is, of the porous
cells), i.e., the translates $\theta (\cdot +h)$ ($h\in \mathbb{Z}^{N}$)
form a relatively compact set in $\ell ^{\infty }(\mathbb{Z}^{N})=L^{\infty
}(\mathbb{Z}^{N},\lambda )$, $\lambda $ being the counting measure. Here the
norm in $\ell ^{\infty }(\mathbb{Z}^{N})$ is given by $\left\Vert
u\right\Vert _{\ell ^{\infty }(\mathbb{Z}^{N})}=\sup_{k\in \mathbb{Z}%
^{N}}\left\vert u(k)\right\vert $ for $u\in \ell ^{\infty }(\mathbb{Z}^{N})$%
. As seen in the remark below, the almost periodicity assumption encompasses
all kinds of periodicity.

\begin{remark}
\label{r1}\emph{1) We recall that the space }$\ell _{AP}^{\infty }(\mathbb{Z}%
^{N})$\emph{\ in which the function }$\theta $\emph{\ lies, is the smallest
closed subspace of }$\ell ^{\infty }(\mathbb{Z}^{N})$\emph{\ that contains
all sequences of the form }%
\begin{equation*}
\mathbb{Z}^{N}\rightarrow \mathbb{C}:k\mapsto \exp (i\mu \cdot k),\ \ \mu
\in \mathbb{R}^{N}.
\end{equation*}%
\emph{Therefore it contains in particular all the periodic sequences defined
on the lattice }$\mathbb{Z}^{N}$\emph{. 2) The function }$\theta $\emph{\
being the characteristic function of a subset of }$\mathbb{Z}^{N}$\emph{,
assuming }$\theta $\emph{\ to be almost periodic in the above sense yields }$%
\theta $\emph{\ is actually periodic. Indeed we find that there is a }$p\in 
\mathbb{Z}^{N}$\emph{\ such that }%
\begin{equation*}
\sup_{k\in \mathbb{Z}^{N}}\left\vert \theta (k+p)-\theta (k)\right\vert <%
\frac{1}{2}.
\end{equation*}%
\emph{Since }$\theta $\emph{\ assumes only }$0$\emph{\ or }$1$\emph{, it
follows that }$\left\vert \theta (k+p)-\theta (k)\right\vert $\emph{\ takes
on either }$0$\emph{\ or }$1$\emph{. We deduce at once that }$\theta
(k+p)-\theta (k)=0$\emph{\ for all }$k\in \mathbb{Z}^{N}$\emph{, so that }$%
\theta $\emph{\ is periodic. Here, what we gain is that }$\theta $\emph{\
may assume any kind of periodicity.}
\end{remark}

\bigskip

Bearing this in mind, let $\Omega $ be an open bounded Lipschitz domain in $%
\mathbb{R}^{N}$ and let $\varepsilon >0$ be a small parameter. We define the
almost periodic porous saturated medium as follows: $\Omega
_{j}^{\varepsilon }=\Omega \cap \varepsilon G_{j}$ ($j=1,2$). The volumes $%
\Omega _{1}^{\varepsilon }$ and $\Omega _{2}^{\varepsilon }$ are
respectively filled by the solid and the fluid, and $\Omega
_{2}^{\varepsilon }$ is connected. We have $\Omega =\Omega _{1}^{\varepsilon
}\cup \Gamma _{12}^{\varepsilon }\cup \Omega _{2}^{\varepsilon }$ (disjoint
union) where $\Gamma _{12}^{\varepsilon }=\partial \Omega _{1}^{\varepsilon
}\cap \partial \Omega _{2}^{\varepsilon }\cap \Omega $ is the part of the
interface of $\Omega _{1}^{\varepsilon }$ with $\Omega _{2}^{\varepsilon }$
lying in $\Omega $. Let $\nu _{j}$ denote the unit outward normal on $%
\partial \Omega _{j}^{\varepsilon }$; then $\nu _{1}=-\nu _{2}$ on $\Gamma
_{12}^{\varepsilon }$.

\subsection{The constitutive equations}

\textit{In the solid}. The elastic constitutive law coupled with the
momentum balance give rise to 
\begin{equation}
\rho _{1}^{\varepsilon }\frac{\partial ^{2}\boldsymbol{u}_{\varepsilon }}{%
\partial t^{2}}-\Div\left[ A_{0}^{\varepsilon }\nabla \boldsymbol{u}%
_{\varepsilon }+\int_{0}^{t}A_{1}^{\varepsilon }(x,t-\tau )\nabla 
\boldsymbol{u}_{\varepsilon }(x,\tau )d\tau \right] =\rho _{1}^{\varepsilon }%
\boldsymbol{f}\text{ in }\Omega _{1}^{\varepsilon }\times (0,T)  \label{1.1}
\end{equation}%
where $\boldsymbol{u}_{\varepsilon }$ is the solid displacement, $%
A_{j}^{\varepsilon }$ the elastic tensors of the material constituting the
skeleton and $\rho _{1}^{\varepsilon }$ the density of that material.

\textit{In the fluid}. The viscoelastic constitutive law associated to the
momentum balance and the incompressibility give rise to 
\begin{equation}
\rho _{2}^{\varepsilon }\frac{\partial \boldsymbol{v}_{\varepsilon }}{%
\partial t}-\Div\left[ B_{0}^{\varepsilon }\nabla \boldsymbol{v}%
_{\varepsilon }+\int_{0}^{t}B_{1}^{\varepsilon }(x,t-\tau )\nabla 
\boldsymbol{v}_{\varepsilon }(x,\tau )d\tau \right] +\nabla p_{\varepsilon
}=\rho _{2}^{\varepsilon }\boldsymbol{g}\text{ in }\Omega _{2}^{\varepsilon
}\times (0,T)  \label{1.2}
\end{equation}%
\begin{equation}
\Div\boldsymbol{v}_{\varepsilon }=0\text{ in }\Omega _{2}^{\varepsilon
}\times (0,T)  \label{1.3}
\end{equation}%
where $\boldsymbol{v}_{\varepsilon }$ denotes the velocity of the fluid, $%
B_{j}^{\varepsilon }$ the dynamical viscosity tensors of the fluid, $\rho
_{2}^{\varepsilon }$ the density of the fluid and $p_{\varepsilon }$ the
pressure of the fluid.

\textit{At the fluid-solid interface}. The adherence condition 
\begin{equation}
\boldsymbol{v}_{\varepsilon }=\frac{\partial \boldsymbol{u}_{\varepsilon }}{%
\partial t}\text{ on }\Gamma _{12}^{\varepsilon }\times (0,T)  \label{1.4}
\end{equation}%
and the normal stress continuity 
\begin{eqnarray}
&&\left( A_{0}^{\varepsilon }\nabla \boldsymbol{u}_{\varepsilon
}+\int_{0}^{t}A_{1}^{\varepsilon }(x,t-\tau )\nabla \boldsymbol{u}%
_{\varepsilon }d\tau \right) \cdot \nu _{1}  \label{1.5} \\
&=&\left( -p_{\varepsilon }I+B_{0}^{\varepsilon }\nabla \boldsymbol{v}%
_{\varepsilon }+\int_{0}^{t}B_{1}^{\varepsilon }(x,t-\tau )\nabla 
\boldsymbol{v}_{\varepsilon }d\tau \right) \cdot \nu _{1}\text{ on }\Gamma
_{12}^{\varepsilon }\times (0,T)  \notag
\end{eqnarray}%
hold.

Finally we supplement Eq. (\ref{1.1})-(\ref{1.5}) by the boundary conditions 
\begin{equation}
\boldsymbol{u}_{\varepsilon }=0\text{ on }\left( \partial \Omega
_{1}^{\varepsilon }\cap \partial \Omega \right) \times (0,T),\ \boldsymbol{v}%
_{\varepsilon }=0\text{ on }\left( \partial \Omega _{2}^{\varepsilon }\cap
\partial \Omega \right) \times (0,T)  \label{1.6}
\end{equation}%
and the initial homogeneous conditions 
\begin{equation}
\boldsymbol{u}_{\varepsilon }(x,0)=\frac{\partial \boldsymbol{u}%
_{\varepsilon }}{\partial t}(x,0)=0\text{ in }\Omega _{1}^{\varepsilon }%
\text{ and }\boldsymbol{v}_{\varepsilon }(x,0)=0\text{ in }\Omega
_{2}^{\varepsilon }.  \label{1.7}
\end{equation}%
Here, $A_{0}^{\varepsilon }(x)=A_{0}\left( x/\varepsilon \right) $ and $%
A_{1}^{\varepsilon }(x,t)=A_{1}\left( x/\varepsilon ,t/\varepsilon \right) $%
, $B_{0}^{\varepsilon }(x)=B_{0}\left( x/\varepsilon \right) $ and $%
B_{1}^{\varepsilon }(x,t)=B_{1}\left( x/\varepsilon ,t/\varepsilon \right) $
($(x,t)\in Q=\Omega \times (0,T)$, $T>0$ a given real number), $\rho
_{j}^{\varepsilon }(x)=\rho _{j}(x/\varepsilon )$ ($j=1,2$) with:

\begin{itemize}
\item[(A1)] $A_{0},B_{0}\in L^{\infty }(\mathbb{R}_{y}^{N})^{N^{2}}$ and $%
A_{1},B_{1}\in L^{\infty }(\mathbb{R}_{y,\tau }^{N+1})^{N^{2}}$ being
symmetric matrices satisfying the following assumption: there exists $\alpha
>0$ such that 
\begin{equation*}
A_{0}\xi \cdot \xi \geq \alpha \left\vert \xi \right\vert ^{2}\text{ and }%
B_{0}\xi \cdot \xi \geq \alpha \left\vert \xi \right\vert ^{2}\text{ a.e. in 
}\Omega \text{ and for all }\xi \in \mathbb{R}^{N}.
\end{equation*}

\item[(A2)] The density function $\rho _{j}$ ($j=1,2$) lies in $L^{\infty }(%
\mathbb{R}_{y}^{N})$ and satisfy $\Lambda ^{-1}\leq \rho _{j}(y)\leq \Lambda 
$ a.e. $y\in \mathbb{R}^{N}$, for some positive $\Lambda $.

\item[(A3)] The functions $\boldsymbol{f}$ and $\boldsymbol{g}$ belong to $%
H^{1}(0,T;L^{2}(\Omega )^{N})$.
\end{itemize}

In view of the assumptions on $A_{l},B_{l}$ and $\rho _{j}$ ($l=0,1$, $j=1,2$%
), the functions $A_{0}^{\varepsilon }$ (resp. $B_{0}^{\varepsilon }$), $%
A_{1}^{\varepsilon }$ (resp. $B_{1}^{\varepsilon }$) and $\rho
_{j}^{\varepsilon }$ are well defined and belong to $L^{\infty }(\Omega
)^{N^{2}}$ (resp. $L^{\infty }(\Omega )^{N^{2}}$), $L^{\infty }(Q)^{N^{2}}$
(resp. $L^{\infty }(Q)^{N^{2}}$) and $L^{\infty }(\Omega )$.

We stress that we may view our model as a \emph{physically correct} one in
the sense that if in Equations (\ref{1.1}) and (\ref{1.3}) (and in the
interface conditions (\ref{1.4}) and (\ref{1.5})) we replace the gradients $%
\nabla \boldsymbol{u}_{\varepsilon }$ and $\nabla \boldsymbol{v}%
_{\varepsilon }$ by the strains $e(\boldsymbol{u}_{\varepsilon })=\frac{1}{2}%
\left( \nabla \boldsymbol{u}_{\varepsilon }+\nabla ^{T}\boldsymbol{u}%
_{\varepsilon }\right) $ and $e(\boldsymbol{v}_{\varepsilon })=\frac{1}{2}%
\left( \nabla \boldsymbol{u}_{\varepsilon }+\nabla ^{T}\boldsymbol{u}%
_{\varepsilon }\right) $ respectively, then thanks to Korn's inequality \cite%
[p. 31]{Pana}, the mathematical analysis does not change, and the model
becomes therefore physically correct. Thus the use of the gradient instead
of the strain is just to simplify the exposition of the results. Also, the
fact that the coefficients $A_{0}$ and $B_{0}$ do not depend on the fast
time variable is just to avoid tricky and useless complications. Indeed
assuming for example $A_{0}$ dependent on $t/\varepsilon $, we will be led
to an estimation containing a term of the order $\varepsilon ^{-1}$.

The model in (\ref{1.1})-(\ref{1.5}) is a generalization of those in \cite%
{AAGM, BA, Dasser}. It seems to be more accurate to study the acoustic
properties of microstructured materials such as glass wool, foam, etc. Let
us precise that here, our considerations differ from those performed in all
the previous work. Indeed, first, in the framework of deterministic
homogenization theory, all the problems considered so far have been studied
in periodic porous media. Here we consider an environment that is almost
periodically perforated, generalizing the results of the earlier works.
Second, the problems solved so far did not involve any memory effect at the
microscale level; see e.g., \cite{MO, Orlik1, Yi}. Here we consider a
problem modeling a phenomenon with a memory effect at the fast
time-variable. From the physical point of view, it means that the memory
effects arising by meeting an obstacle decay in the surrounding of the next
obstacle. So considered, the problem is new and the obtained results are of
physical interest. Finally, though the problem considered here is linear, we
do not use the Laplace transform to achieve our goal. We rather use a direct
method, the \emph{two-scale convolution} (in the general framework of $%
\Sigma $\emph{-convergence}), leaning on some recent results about the
convergence of sequences defined by the convolution; see e.g., Theorem \ref%
{t2.6'}. Let us precise that the periodic counterpart of Theorem \ref{t2.6'}
is already known and has been proved by Visintin \cite[Proposition 2.13]%
{Visintin} by using the periodic unfolding method. A more general version of
that result in the framework of $\Sigma $-convergence has been for the first
time stated and proved in \cite{SW}, but with the restricted assumption that
the open set $\Omega \subset \mathbb{R}^{N}$ is bounded. Still in the $%
\Sigma $-convergence framework, a more general proof (in any open set $%
\Omega \subset \mathbb{R}^{N}$) has been given in the almost periodic
setting in a very recent work \cite{M2AS}. The use of this method instead of
the Laplace transform method is motivated by at least two facts: 1) It
enables one to tackle both linear and nonlinear homogenization problems
involving convolution both in fast space and time variables (see \cite{SW,
Introverted}), which is not the case for the latter method; 2) As observed
in numerical experiments by the authors of \cite{Yi} (see the Section 5
therein), "\emph{the memory effects make numerical inverse Laplace transform
very complicated}".

Our aim is to pass to the limit as the size $\varepsilon $ of the
microstructures shrinks to $0$ in the above system, and of course when the
elasticity and viscosity tensors, the density of both fluid and skeleton,
are subjected to be almost periodic, together with the assumption that the
medium is almost periodically perforated. Based on the above considerations,
we may therefore say that our results are new.

The work is framed as follows. Section 2 deals with the statement and the
proof of the existence result for problem (\ref{1.1})-(\ref{1.7}). We also
derive therein some a priori estimates and a compactness result. In section
3, we give some fundamental facts about almost periodicity. The sigma
convergence concept with some related convolution results are discussed in
Section 4. In Section 5, we derive the homogenization result leading to the
properties of the equivalent medium as stated earlier.

Throughout the work, vector spaces are assumed to be real vector spaces, and
scalar functions are assumed to take real values. We shall always assume
that the numerical space $\mathbb{R}^{N}$ (integer $N\geq 1$) and its open
sets are each equipped with the Lebesgue measure $dx=dx_{1}...dx_{N}$.

\section{Existence and compactness results}

We begin by giving the functional setting of our problem. We define the
function spaces 
\begin{equation*}
V_{\varepsilon }^{j}=\{\left. \boldsymbol{v}\right\vert _{\Omega
_{j}^{\varepsilon }}:v\in H_{0}^{1}(\Omega )^{N}\}\text{ with the norm }%
\left\Vert \boldsymbol{v}\right\Vert _{V_{\varepsilon }^{j}}=\left\Vert
\nabla v\right\Vert _{L^{2}(\Omega _{j}^{\varepsilon })^{N^{2}}}\text{, }%
j=1,2\text{;}
\end{equation*}%
\begin{equation*}
W=\{\boldsymbol{v}\in H_{0}^{1}(\Omega )^{N}:\Div\boldsymbol{v}=0\text{ in }%
\Omega _{2}^{\varepsilon }\}\text{ with usual }H_{0}^{1}(\Omega )^{N}\text{%
-norm;}
\end{equation*}%
\begin{equation*}
H=\text{ the closure of }W\text{ for the norm induced by the inner product }%
\left( \left( \cdot ,\cdot \right) \right) ,
\end{equation*}%
where $\left( \left( \cdot ,\cdot \right) \right) $ denotes the weighted $%
L^{2}$ inner product 
\begin{equation*}
\left( \left( \boldsymbol{u},\boldsymbol{v}\right) \right) =\int_{\Omega
_{1}^{\varepsilon }}\rho _{1}^{\varepsilon }\boldsymbol{u}\cdot \boldsymbol{v%
}dx+\int_{\Omega _{2}^{\varepsilon }}\rho _{2}^{\varepsilon }\boldsymbol{u}%
\cdot \boldsymbol{v}dx\text{ for }\boldsymbol{u},\boldsymbol{v}\in
L^{2}(\Omega )^{N}.
\end{equation*}%
The induced norm on $H$ is equivalent to the $L^{2}$-norm and $H$ preserves
the divergence free property in $\Omega _{2}^{\varepsilon }$. We denote by $%
\left\langle \left\langle \cdot ,\cdot \right\rangle \right\rangle $ the
duality pairing between $W^{\prime }$ and $W$ generated from the above inner
product $\left( \left( \cdot ,\cdot \right) \right) $. This being so, we
have the evolution triple $W\hookrightarrow H\hookrightarrow W^{\prime }$
where we identify $H$ with its topological dual $H^{\prime }$.

In order to solve (\ref{1.1})-(\ref{1.7}), we need an auxiliary equivalent
formulation of the said-problem. Prior to that, let us introduce the
auxiliary functions 
\begin{equation*}
\boldsymbol{w}_{\varepsilon }(t)=\int_{0}^{t}\boldsymbol{v}_{\varepsilon
}(s)ds,\ \boldsymbol{u}^{\varepsilon }=\chi _{1}^{\varepsilon }\boldsymbol{u}%
_{\varepsilon }+\chi _{2}^{\varepsilon }\boldsymbol{w}_{\varepsilon }\text{, 
}\rho ^{\varepsilon }=\chi _{1}^{\varepsilon }\rho _{1}^{\varepsilon }+\chi
_{2}^{\varepsilon }\rho _{2}^{\varepsilon }\text{ and }F_{\varepsilon }=\chi
_{1}^{\varepsilon }\boldsymbol{f}+\chi _{2}^{\varepsilon }\boldsymbol{g}
\end{equation*}%
where $\chi _{j}^{\varepsilon }$ ($j=1,2$) stands for the characteristic
function of the set $\Omega _{j}^{\varepsilon }$ in $\Omega $. It is easy to
see that $\chi _{j}^{\varepsilon }(x)=\chi _{j}(x/\varepsilon )$ ($x\in
\Omega _{j}^{\varepsilon }$), where $\chi _{j}$ is the characteristic
function of the set $G_{j}$ in $\mathbb{R}^{N}$. With the above notation, we
may reformulate system (\ref{1.1})-(\ref{1.7}) as a variational problem,
that is, find $\boldsymbol{u}^{\varepsilon }\in \mathcal{C}%
^{1}([0,T];L^{2}(\Omega )^{N})\cap \mathcal{C}([0,T];H_{0}^{1}(\Omega )^{N})$
and $p_{\varepsilon }\in L^{2}(0,T;L^{2}(\Omega _{2}^{\varepsilon }))$ such
that 
\begin{eqnarray}
&&\frac{d^{2}}{dt^{2}}\left( \left( \boldsymbol{u}^{\varepsilon }(t),\varphi
\right) \right) \text{+}\int_{\Omega }\left( \chi _{1}^{\varepsilon
}A_{0}^{\varepsilon }\nabla \boldsymbol{u}^{\varepsilon }(t)\text{+}\chi
_{2}^{\varepsilon }B_{0}^{\varepsilon }\nabla \frac{\partial \boldsymbol{u}%
^{\varepsilon }}{\partial t}(t)\right) \cdot \nabla \varphi dx\text{-}%
\int_{\Omega }\chi _{2}^{\varepsilon }p_{\varepsilon }\Div\varphi dx
\label{1.8'} \\
&&+\int_{0}^{t}\left( \int_{\Omega }\left( \chi _{1}^{\varepsilon
}A_{1}^{\varepsilon }(t-\tau )\nabla \boldsymbol{u}^{\varepsilon }(\tau
)+\chi _{2}^{\varepsilon }B_{1}^{\varepsilon }(t-\tau )\nabla \frac{\partial 
\boldsymbol{u}^{\varepsilon }}{\partial t}(\tau )\right) \cdot \nabla
\varphi dx\right) d\tau  \notag \\
&=&\left( \left( F_{\varepsilon },\varphi \right) \right) \text{ for all }%
\varphi \in H_{0}^{1}(\Omega )^{N}.  \notag
\end{eqnarray}%
The choice of a global test function $\varphi \in H_{0}^{1}(\Omega )^{N}$
instead of two independent test functions on the two subdomains allows us to
incorporate the interface conditions (\ref{1.4})-(\ref{1.5}) into the weak
formulation. Equation (\ref{1.8'}) requires $\frac{\partial \boldsymbol{u}%
^{\varepsilon }}{\partial t}\in L^{2}(0,T;L^{2}(\Omega )^{N})$, may be
written under the equivalent divergence-free weak form 
\begin{eqnarray*}
&&\frac{d^{2}}{dt^{2}}\left( \left( \boldsymbol{u}^{\varepsilon }(t),\varphi
\right) \right) +\int_{\Omega }\left( \chi _{1}^{\varepsilon
}A_{0}^{\varepsilon }\nabla \boldsymbol{u}^{\varepsilon }(t)+\chi
_{2}^{\varepsilon }B_{0}^{\varepsilon }\nabla \frac{\partial \boldsymbol{u}%
^{\varepsilon }}{\partial t}(t)\right) \cdot \nabla \varphi dx \\
&&+\int_{0}^{t}\left( \int_{\Omega }\left( \chi _{1}^{\varepsilon
}A_{1}^{\varepsilon }(t-\tau )\nabla \boldsymbol{u}^{\varepsilon }(\tau
)+\chi _{2}^{\varepsilon }B_{1}^{\varepsilon }(t-\tau )\nabla \frac{\partial 
\boldsymbol{u}^{\varepsilon }}{\partial t}(\tau )\right) \cdot \nabla
\varphi dx\right) d\tau \\
&=&\left( \left( F_{\varepsilon },\varphi \right) \right) \text{, all }%
\varphi \in W,
\end{eqnarray*}%
or equivalently 
\begin{equation}
\frac{d^{2}}{dt^{2}}\left( \left( \boldsymbol{u}^{\varepsilon }(t),\varphi
\right) \right) +\left\langle \left\langle \mathcal{A}_{\varepsilon }(t)%
\boldsymbol{u}^{\varepsilon }(t),\varphi \right\rangle \right\rangle
+\int_{0}^{t}\left\langle \left\langle \mathcal{B}_{\varepsilon }(t-\tau )%
\boldsymbol{u}^{\varepsilon }(\tau ),\varphi \right\rangle \right\rangle
d\tau =\left( \left( F_{\varepsilon },\varphi \right) \right) \text{ }%
\forall \varphi \in W,  \label{1.10}
\end{equation}%
$\mathcal{A}_{\varepsilon }(t)$ and $\mathcal{B}_{\varepsilon }(t-\tau )$
(for $0\leq \tau \leq t\leq T$) being bounded operators defined from $W$
into $W^{\prime }$ by 
\begin{equation*}
\left\langle \left\langle \mathcal{A}_{\varepsilon }(t)\boldsymbol{u}%
^{\varepsilon }(t),\varphi \right\rangle \right\rangle =\int_{\Omega }\left(
\chi _{1}^{\varepsilon }A_{0}^{\varepsilon }\nabla \boldsymbol{u}%
^{\varepsilon }(t)+\chi _{2}^{\varepsilon }B_{0}^{\varepsilon }\nabla \frac{%
\partial \boldsymbol{u}^{\varepsilon }}{\partial t}(t)\right) \cdot \nabla
\varphi dx
\end{equation*}%
\begin{equation*}
\left\langle \left\langle \mathcal{B}_{\varepsilon }(t-\tau )\boldsymbol{u}%
^{\varepsilon }(\tau ),\varphi \right\rangle \right\rangle =\int_{\Omega
}\left( \chi _{1}^{\varepsilon }A_{1}^{\varepsilon }(t-\tau )\nabla 
\boldsymbol{u}^{\varepsilon }(\tau )+\chi _{2}^{\varepsilon
}B_{1}^{\varepsilon }(t-\tau )\nabla \frac{\partial \boldsymbol{u}%
^{\varepsilon }}{\partial t}(\tau )\right) \cdot \nabla \varphi dx.
\end{equation*}

In order to ensure the existence and uniqueness of $\boldsymbol{u}%
^{\varepsilon }$ we use respectively \cite[Theorem 3.4]{DGHL2} and \cite[%
Theorem 3.2]{Orlik} as follows. First, in view of \cite[Theorem 3.4]{DGHL2},
the operator $\frac{d^{2}}{dt^{2}}+\mathcal{A}_{\varepsilon }$ is one-to-one
from $L^{2}(0,T;W)$ onto $L^{2}(0,T;W^{\prime })$. Next since $\mathcal{B}%
_{\varepsilon }(t-\tau )$ is bounded from $W$ into $W^{\prime }$, it readily
follows from \cite[Theorem 3.2]{Orlik} that (\ref{1.10}) possesses a unique
solution $\boldsymbol{u}^{\varepsilon }\in L^{2}(0,T;W)$ with $\frac{d^{2}%
\boldsymbol{u}^{\varepsilon }}{dt^{2}}\in L^{2}(0,T;W^{\prime })$. Therefore
by setting $\boldsymbol{u}_{\varepsilon }=\chi _{1}^{\varepsilon }%
\boldsymbol{u}^{\varepsilon }$ and $\boldsymbol{v}_{\varepsilon }=\chi
_{2}^{\varepsilon }\frac{\partial \boldsymbol{u}^{\varepsilon }}{\partial t}$%
, we readily obtain a solution $(\boldsymbol{u}_{\varepsilon },\boldsymbol{v}%
_{\varepsilon })\in L^{2}(0,T;V_{\varepsilon }^{1})\times
L^{2}(0,T;V_{\varepsilon }^{2})$ for the divergence-free formulation of (\ref%
{1.1})-(\ref{1.7}), and owing to the regularity of the data (see assumption
(A3)), the following result is in order.

\begin{theorem}
\label{t2.1}For each $\varepsilon >0$ and under Assumptions \emph{(A1)-(A3)}%
, there exists a unique triplet $(\boldsymbol{u}_{\varepsilon },\boldsymbol{v%
}_{\varepsilon },p_{\varepsilon })$ which possesses the regularity $%
\boldsymbol{u}_{\varepsilon }\in W^{1,\infty }(0,T;V_{\varepsilon }^{1}),\ 
\frac{\partial ^{2}\boldsymbol{u}_{\varepsilon }}{\partial t^{2}}\in
L^{\infty }(0,T;L^{2}(\Omega _{1}^{\varepsilon })^{N}),\ \boldsymbol{v}%
_{\varepsilon }\in W^{1,\infty }(0,T;H_{\varepsilon })\cap
L^{2}(0,T;V_{\varepsilon }^{2})$, $p_{\varepsilon }\in L^{\infty
}(0,T;L^{2}(\Omega _{2}^{\varepsilon }))$, and satisfies equations \emph{(%
\ref{1.1})-(\ref{1.7})}. Moreover, $\boldsymbol{u}_{\varepsilon }\in 
\mathcal{C}^{1}([0,T];L^{2}(\Omega _{1}^{\varepsilon })^{N})$, $\boldsymbol{v%
}_{\varepsilon }\in \mathcal{C}([0,T];H_{\varepsilon })$ and the following
estimates hold: 
\begin{equation}
\sup_{0\leq t\leq T}\left( \left\Vert \boldsymbol{u}_{\varepsilon
}(t)\right\Vert _{L^{2}(\Omega _{1}^{\varepsilon })}^{2}+\left\Vert \nabla 
\boldsymbol{u}_{\varepsilon }(t)\right\Vert _{L^{2}(\Omega _{1}^{\varepsilon
})}^{2}+\left\Vert \boldsymbol{v}_{\varepsilon }(t)\right\Vert
_{L^{2}(\Omega _{2}^{\varepsilon })}^{2}\right) \leq C,  \label{2.1}
\end{equation}%
\begin{equation}
\underset{0\leq t\leq T}{ess\sup }\left( \left\Vert \frac{\partial 
\boldsymbol{v}_{\varepsilon }}{\partial t}(t)\right\Vert _{L^{2}(\Omega
_{2}^{\varepsilon })}^{2}+\left\Vert \frac{\partial ^{2}\boldsymbol{u}%
_{\varepsilon }}{\partial t^{2}}(t)\right\Vert _{L^{2}(\Omega
_{1}^{\varepsilon })}^{2}\right) \leq C,  \label{2.2}
\end{equation}%
\begin{equation}
\sup_{0\leq t\leq T}\left( \left\Vert \frac{\partial \boldsymbol{u}%
_{\varepsilon }}{\partial t}(t)\right\Vert _{L^{2}(\Omega _{1}^{\varepsilon
})}^{2}\right) \leq C\text{ and }\int_{0}^{T}\left\Vert \nabla \boldsymbol{v}%
_{\varepsilon }(t)\right\Vert _{L^{2}(\Omega _{2}^{\varepsilon })}^{2}dt\leq
C,  \label{2.31}
\end{equation}%
\begin{equation}
\left\Vert p_{\varepsilon }\right\Vert _{L^{2}(\Omega _{2}^{\varepsilon
}\times (0,T))}\leq C  \label{2.3}
\end{equation}%
where $C$ is a positive constant not depending upon $\varepsilon $.
\end{theorem}

\begin{proof}
It just remains to prove the existence and uniqueness of the pressure,
together with the uniform estimates. But the existence and uniqueness of $%
p_{\varepsilon }$ is derived exactly as in \cite[Section 3.2]{DGHL2} (see
especially Theorem 3.4 therein).

It remains to verify inequalities (\ref{2.1})-(\ref{2.3}). Multiply Eq. (\ref%
{1.1}) by $\frac{\partial \boldsymbol{u}_{\varepsilon }}{\partial t}(t)=%
\boldsymbol{u}_{\varepsilon }^{\prime }(t)$, Eq. (\ref{1.2}) by $\boldsymbol{%
v}_{\varepsilon }(t)$ and sum up, then 
\begin{equation}
\begin{array}{l}
\left( \rho _{1}^{\varepsilon }\frac{\partial ^{2}\boldsymbol{u}%
_{\varepsilon }}{\partial t^{2}},\frac{\partial \boldsymbol{u}_{\varepsilon }%
}{\partial t}\right) _{L^{2}(\Omega _{1}^{\varepsilon })}+\left( \rho
_{2}^{\varepsilon }\frac{\partial \boldsymbol{v}_{\varepsilon }}{\partial t},%
\boldsymbol{v}_{\varepsilon }\right) _{L^{2}(\Omega _{2}^{\varepsilon
})}+\left( A_{0}^{\varepsilon }\nabla \boldsymbol{u}_{\varepsilon
}+A_{1}^{\varepsilon }\ast \nabla \boldsymbol{u}_{\varepsilon },\nabla 
\boldsymbol{u}_{\varepsilon }^{\prime }\right) _{L^{2}(\Omega
_{1}^{\varepsilon })} \\ 
+\left( B_{0}^{\varepsilon }\nabla \boldsymbol{u}_{\varepsilon
}+B_{1}^{\varepsilon }\ast \nabla \boldsymbol{u}_{\varepsilon },\nabla 
\boldsymbol{v}_{\varepsilon }\right) _{L^{2}(\Omega _{2}^{\varepsilon
})}=\left( \rho _{1}^{\varepsilon }\boldsymbol{f},\boldsymbol{u}%
_{\varepsilon }^{\prime }\right) _{L^{2}(\Omega _{1}^{\varepsilon })}+\left(
\rho _{2}^{\varepsilon }\boldsymbol{g},\boldsymbol{v}_{\varepsilon }\right)
_{L^{2}(\Omega _{2}^{\varepsilon })}%
\end{array}
\label{2.4}
\end{equation}%
where $(,)_{L^{2}}$ stands for the inner product in $L^{2}$. But (\ref{2.4})
is equivalent to 
\begin{equation}
\begin{array}{l}
\frac{1}{2}\frac{d}{dt}\left\Vert (\rho _{1}^{\varepsilon })^{\frac{1}{2}}%
\boldsymbol{u}_{\varepsilon }^{\prime }(t)\right\Vert _{L^{2}(\Omega
_{1}^{\varepsilon })}^{2}+\frac{1}{2}\frac{d}{dt}\left\Vert (\rho
_{2}^{\varepsilon })^{\frac{1}{2}}\boldsymbol{v}_{\varepsilon
}(t)\right\Vert _{L^{2}(\Omega _{2}^{\varepsilon })}^{2} \\ 
+\frac{1}{2}\frac{d}{dt}\left( A_{0}^{\varepsilon }\nabla \boldsymbol{u}%
_{\varepsilon }+A_{1}^{\varepsilon }\ast \nabla \boldsymbol{u}_{\varepsilon
},\nabla \boldsymbol{u}_{\varepsilon }\right) _{L^{2}(\Omega
_{1}^{\varepsilon })}+\left( B_{0}^{\varepsilon }\nabla \boldsymbol{v}%
_{\varepsilon }+B_{1}^{\varepsilon }\ast \nabla \boldsymbol{v}_{\varepsilon
},\nabla \boldsymbol{v}_{\varepsilon }\right) _{L^{2}(\Omega
_{2}^{\varepsilon })} \\ 
\ \ =\left( \rho _{1}^{\varepsilon }\boldsymbol{f},\boldsymbol{u}%
_{\varepsilon }^{\prime }(t)\right) _{L^{2}(\Omega _{1}^{\varepsilon
})}+\left( \rho _{2}^{\varepsilon }\boldsymbol{g},\boldsymbol{v}%
_{\varepsilon }(t)\right) _{L^{2}(\Omega _{2}^{\varepsilon })}.%
\end{array}
\label{2.5}
\end{equation}%
Integrate (\ref{2.5}) with respect to $t$, 
\begin{equation*}
\begin{array}{l}
\left\Vert (\rho _{1}^{\varepsilon })^{\frac{1}{2}}\boldsymbol{u}%
_{\varepsilon }^{\prime }(t)\right\Vert _{L^{2}(\Omega _{1}^{\varepsilon
})}^{2}+\left\Vert (\rho _{2}^{\varepsilon })^{\frac{1}{2}}\boldsymbol{v}%
_{\varepsilon }(t)\right\Vert _{L^{2}(\Omega _{2}^{\varepsilon })}^{2} \\ 
+\left( A_{0}^{\varepsilon }\nabla \boldsymbol{u}_{\varepsilon
}+A_{1}^{\varepsilon }\ast \nabla \boldsymbol{u}_{\varepsilon },\nabla 
\boldsymbol{u}_{\varepsilon }\right) _{L^{2}(\Omega _{1}^{\varepsilon
})}+2\int_{0}^{t}\left( B_{0}^{\varepsilon }\nabla \boldsymbol{v}%
_{\varepsilon }+B_{1}^{\varepsilon }\ast \nabla \boldsymbol{v}_{\varepsilon
},\nabla \boldsymbol{v}_{\varepsilon }\right) _{L^{2}(\Omega
_{2}^{\varepsilon })}d\tau \\ 
\ \ \ \ \ =2\int_{0}^{t}\left[ \left( \rho _{1}^{\varepsilon }\boldsymbol{f},%
\boldsymbol{u}_{\varepsilon }^{\prime }(t)\right) _{L^{2}(\Omega
_{1}^{\varepsilon })}+\left( \rho _{2}^{\varepsilon }\boldsymbol{g},%
\boldsymbol{v}_{\varepsilon }(t)\right) _{L^{2}(\Omega _{2}^{\varepsilon })}%
\right] d\tau .%
\end{array}%
\end{equation*}%
But 
\begin{eqnarray}
&&2\int_{0}^{t}\left[ \left( \rho _{1}^{\varepsilon }\boldsymbol{f},%
\boldsymbol{u}_{\varepsilon }^{\prime }(t)\right) _{L^{2}(\Omega
_{1}^{\varepsilon })}+\left( \rho _{2}^{\varepsilon }\boldsymbol{g},%
\boldsymbol{v}_{\varepsilon }(t)\right) _{L^{2}(\Omega _{2}^{\varepsilon })}%
\right] d\tau  \label{2.6} \\
&\leq &\int_{0}^{t}\left( \left\Vert (\rho _{1}^{\varepsilon })^{\frac{1}{2}}%
\boldsymbol{f}(\tau )\right\Vert _{L^{2}(\Omega _{1}^{\varepsilon
})}^{2}+\left\Vert (\rho _{2}^{\varepsilon })^{\frac{1}{2}}\boldsymbol{g}%
(\tau )\right\Vert _{L^{2}(\Omega _{2}^{\varepsilon })}^{2}\right) d\tau 
\notag \\
&&+\int_{0}^{t}\left( \left\Vert (\rho _{1}^{\varepsilon })^{\frac{1}{2}}%
\boldsymbol{u}_{\varepsilon }^{\prime }(\tau )\right\Vert _{L^{2}(\Omega
_{1}^{\varepsilon })}^{2}+\left\Vert (\rho _{2}^{\varepsilon })^{\frac{1}{2}}%
\boldsymbol{v}_{\varepsilon }(\tau )\right\Vert _{L^{2}(\Omega
_{2}^{\varepsilon })}^{2}\right) d\tau .  \notag
\end{eqnarray}%
Making use of (A1) together with (\ref{2.5})-(\ref{2.6}) we obtain 
\begin{equation}
\begin{array}{l}
\left\Vert (\rho _{1}^{\varepsilon })^{\frac{1}{2}}\boldsymbol{u}%
_{\varepsilon }^{\prime }(t)\right\Vert _{L^{2}(\Omega _{1}^{\varepsilon
})}^{2}+\left\Vert (\rho _{2}^{\varepsilon })^{\frac{1}{2}}\boldsymbol{v}%
_{\varepsilon }(t)\right\Vert _{L^{2}(\Omega _{2}^{\varepsilon
})}^{2}+\alpha \left\Vert \nabla \boldsymbol{u}_{\varepsilon }(t)\right\Vert
_{L^{2}(\Omega _{1}^{\varepsilon })}^{2}+2\alpha \int_{0}^{t}\left\Vert
\nabla \boldsymbol{v}_{\varepsilon }(\tau )\right\Vert _{L^{2}(\Omega
_{2}^{\varepsilon })}^{2}d\tau \\ 
\leq \int_{0}^{T}\left( \left\Vert (\rho _{1}^{\varepsilon })^{\frac{1}{2}}%
\boldsymbol{f}(\tau )\right\Vert _{L^{2}(\Omega _{1}^{\varepsilon
})}^{2}+\left\Vert (\rho _{2}^{\varepsilon })^{\frac{1}{2}}\boldsymbol{g}%
(\tau )\right\Vert _{L^{2}(\Omega _{2}^{\varepsilon })}^{2}\right) d\tau
-\left( A_{1}^{\varepsilon }\ast \nabla \boldsymbol{u}_{\varepsilon },\nabla 
\boldsymbol{u}_{\varepsilon }\right) _{L^{2}(\Omega _{1}^{\varepsilon })} \\ 
\ \ \ -2\int_{0}^{t}\left( B_{1}^{\varepsilon }\ast \nabla \boldsymbol{v}%
_{\varepsilon },\nabla \boldsymbol{v}_{\varepsilon }\right) _{L^{2}(\Omega
_{2}^{\varepsilon })}d\tau +\int_{0}^{t}\left( \left\Vert (\rho
_{1}^{\varepsilon })^{\frac{1}{2}}\boldsymbol{u}_{\varepsilon }^{\prime
}(\tau )\right\Vert _{L^{2}(\Omega _{1}^{\varepsilon })}^{2}+\left\Vert
(\rho _{2}^{\varepsilon })^{\frac{1}{2}}\boldsymbol{v}_{\varepsilon }(\tau
)\right\Vert _{L^{2}(\Omega _{2}^{\varepsilon })}^{2}\right) d\tau ,%
\end{array}
\label{2.7}
\end{equation}%
where here the convolution $\ast $ is taken with respect to $t$. But using
Young's inequality, 
\begin{equation*}
\begin{array}{l}
\left( A_{1}^{\varepsilon }\ast \nabla \boldsymbol{u}_{\varepsilon },\nabla 
\boldsymbol{u}_{\varepsilon }\right) _{L^{2}(\Omega _{1}^{\varepsilon })} \\ 
=\left( \int_{0}^{t}A_{1}^{\varepsilon }(\cdot ,t-\tau )\nabla \boldsymbol{u}%
_{\varepsilon }(\tau )d\tau ,\nabla \boldsymbol{u}_{\varepsilon }(t)\right)
_{L^{2}(\Omega _{1}^{\varepsilon })} \\ 
=\int_{0}^{t}\left( \int_{\Omega _{1}^{\varepsilon }}A_{1}^{\varepsilon
}(\cdot ,t-\tau )\nabla \boldsymbol{u}_{\varepsilon }(\tau )\cdot \nabla 
\boldsymbol{u}_{\varepsilon }(t)dx\right) d\tau \\ 
\leq \int_{0}^{t}\left( C\left\Vert A_{1}^{\varepsilon }(t-\tau )\nabla 
\boldsymbol{u}_{\varepsilon }(\tau )\right\Vert _{L^{2}(\Omega
_{1}^{\varepsilon })}^{2}+\frac{\alpha }{2T}\left\Vert \nabla \boldsymbol{u}%
_{\varepsilon }(t)\right\Vert _{L^{2}(\Omega _{1}^{\varepsilon
})}^{2}\right) d\tau \\ 
\leq \frac{\alpha }{2}\left\Vert \nabla \boldsymbol{u}_{\varepsilon
}(t)\right\Vert _{L^{2}(\Omega _{1}^{\varepsilon
})}^{2}+C\int_{0}^{t}\left\Vert A_{1}^{\varepsilon }(t-\tau )\nabla 
\boldsymbol{u}_{\varepsilon }(\tau )\right\Vert _{L^{2}(\Omega
_{1}^{\varepsilon })}^{2}d\tau \\ 
\leq \frac{\alpha }{2}\left\Vert \nabla \boldsymbol{u}_{\varepsilon
}(t)\right\Vert _{L^{2}(\Omega _{1}^{\varepsilon
})}^{2}+C\int_{0}^{t}\left\Vert \nabla \boldsymbol{u}_{\varepsilon }(\tau
)\right\Vert _{L^{2}(\Omega _{1}^{\varepsilon })}^{2}d\tau ,%
\end{array}%
\end{equation*}%
hence 
\begin{equation}
\left( A_{1}^{\varepsilon }\ast \nabla \boldsymbol{u}_{\varepsilon },\nabla 
\boldsymbol{u}_{\varepsilon }\right) _{L^{2}(\Omega _{1}^{\varepsilon
})}\leq \frac{\alpha }{2}\left\Vert \nabla \boldsymbol{u}_{\varepsilon
}(t)\right\Vert _{L^{2}(\Omega _{1}^{\varepsilon
})}^{2}+C\int_{0}^{t}\left\Vert \nabla \boldsymbol{u}_{\varepsilon }(\tau
)\right\Vert _{L^{2}(\Omega _{1}^{\varepsilon })}^{2}d\tau .  \label{2.8}
\end{equation}%
Also 
\begin{equation*}
\begin{array}{l}
2\int_{0}^{t}\left( (B_{1}^{\varepsilon }\ast \nabla \boldsymbol{v}%
_{\varepsilon })(\tau ),\nabla \boldsymbol{v}_{\varepsilon }(\tau )\right)
_{L^{2}(\Omega _{2}^{\varepsilon })}d\tau \\ 
=2\int_{0}^{t}\left( \int_{0}^{\tau }\left( \int_{\Omega _{2}^{\varepsilon
}}B_{1}^{\varepsilon }(\tau -s)\nabla \boldsymbol{v}_{\varepsilon }(s)\cdot
\nabla \boldsymbol{v}_{\varepsilon }(\tau )dx\right) ds\right) d\tau \\ 
\leq 2\int_{0}^{t}\left( C\int_{0}^{\tau }\left\Vert B_{1}^{\varepsilon
}(\tau -s)\nabla \boldsymbol{v}_{\varepsilon }(s)\right\Vert _{L^{2}(\Omega
_{2}^{\varepsilon })}^{2}ds+\int_{0}^{\tau }\frac{\alpha }{2\tau }\left\Vert
\nabla \boldsymbol{v}_{\varepsilon }(\tau )\right\Vert _{L^{2}(\Omega
_{2}^{\varepsilon })}^{2}ds\right) d\tau \\ 
\leq \alpha \int_{0}^{t}\left\Vert \nabla \boldsymbol{v}_{\varepsilon }(\tau
)\right\Vert _{L^{2}(\Omega _{2}^{\varepsilon })}^{2}d\tau
+C\int_{0}^{t}\left( \int_{0}^{\tau }\left\Vert \nabla \boldsymbol{v}%
_{\varepsilon }(s)\right\Vert _{L^{2}(\Omega _{2}^{\varepsilon
})}^{2}ds\right) d\tau ,%
\end{array}%
\end{equation*}%
hence 
\begin{eqnarray}
&&2\int_{0}^{t}\left( (B_{1}^{\varepsilon }\ast \nabla \boldsymbol{v}%
_{\varepsilon })(\tau ),\nabla \boldsymbol{v}_{\varepsilon }(\tau )\right)
_{L^{2}(\Omega _{2}^{\varepsilon })}d\tau  \label{2.9} \\
&\leq &\alpha \int_{0}^{t}\left\Vert \nabla \boldsymbol{v}_{\varepsilon
}(\tau )\right\Vert _{L^{2}(\Omega _{2}^{\varepsilon })}^{2}d\tau
+C\int_{0}^{t}\left( \int_{0}^{\tau }\left\Vert \nabla \boldsymbol{v}%
_{\varepsilon }(s)\right\Vert _{L^{2}(\Omega _{2}^{\varepsilon
})}^{2}ds\right) d\tau .  \notag
\end{eqnarray}%
Putting together (\ref{2.8}) and (\ref{2.9}), it emerges from (\ref{2.7})
that 
\begin{equation*}
\begin{array}{l}
\left\Vert (\rho _{1}^{\varepsilon })^{\frac{1}{2}}\boldsymbol{u}%
_{\varepsilon }^{\prime }(t)\right\Vert _{L^{2}(\Omega _{1}^{\varepsilon
})}^{2}+\left\Vert (\rho _{2}^{\varepsilon })^{\frac{1}{2}}\boldsymbol{v}%
_{\varepsilon }(t)\right\Vert _{L^{2}(\Omega _{2}^{\varepsilon })}^{2}+\frac{%
\alpha }{2}\left\Vert \nabla \boldsymbol{u}_{\varepsilon }(t)\right\Vert
_{L^{2}(\Omega _{1}^{\varepsilon })}^{2}+\alpha \int_{0}^{t}\left\Vert
\nabla \boldsymbol{v}_{\varepsilon }(\tau )\right\Vert _{L^{2}(\Omega
_{2}^{\varepsilon })}^{2}d\tau \\ 
\leq C+\int_{0}^{t}[\left\Vert (\rho _{1}^{\varepsilon })^{\frac{1}{2}}%
\boldsymbol{u}_{\varepsilon }^{\prime }(\tau )\right\Vert _{L^{2}(\Omega
_{1}^{\varepsilon })}^{2}+\left\Vert (\rho _{2}^{\varepsilon })^{\frac{1}{2}}%
\boldsymbol{v}_{\varepsilon }(\tau )\right\Vert _{L^{2}(\Omega
_{2}^{\varepsilon })}^{2} \\ 
\ \ \ \ \ \ \ \ \ \ \ \ \ \ \ \ \ \ \ \ \ \ \ \ \ \ \ \ \ +C\left\Vert
\nabla \boldsymbol{u}_{\varepsilon }(\tau )\right\Vert _{L^{2}(\Omega
_{1}^{\varepsilon })}^{2}+C\int_{0}^{\tau }\left\Vert \nabla \boldsymbol{v}%
_{\varepsilon }(s)\right\Vert _{L^{2}(\Omega _{2}^{\varepsilon
})}^{2}ds]d\tau .%
\end{array}%
\end{equation*}%
Appealing to Gronwall's inequality, 
\begin{equation*}
\left\Vert (\rho _{1}^{\varepsilon })^{\frac{1}{2}}\boldsymbol{u}%
_{\varepsilon }^{\prime }(t)\right\Vert _{L^{2}(\Omega _{1}^{\varepsilon
})}^{2}+\left\Vert (\rho _{2}^{\varepsilon })^{\frac{1}{2}}\boldsymbol{v}%
_{\varepsilon }(t)\right\Vert _{L^{2}(\Omega _{2}^{\varepsilon
})}^{2}+\left\Vert \nabla \boldsymbol{u}_{\varepsilon }(t)\right\Vert
_{L^{2}(\Omega _{1}^{\varepsilon })}^{2}+\int_{0}^{t}\left\Vert \nabla 
\boldsymbol{v}_{\varepsilon }(\tau )\right\Vert _{L^{2}(\Omega
_{2}^{\varepsilon })}^{2}d\tau \leq C
\end{equation*}%
for all $0\leq t\leq T$ and all $\varepsilon >0$. We recall that in all the
above inequalities, $C$ is a positive constant independent of both $%
\varepsilon $ and $t$. We therefore infer from (A2) that 
\begin{equation}
\left\Vert \boldsymbol{u}_{\varepsilon }^{\prime }(t)\right\Vert
_{L^{2}(\Omega _{1}^{\varepsilon })}^{2}+\left\Vert \boldsymbol{v}%
_{\varepsilon }(t)\right\Vert _{L^{2}(\Omega _{2}^{\varepsilon
})}^{2}+\left\Vert \nabla \boldsymbol{u}_{\varepsilon }(t)\right\Vert
_{L^{2}(\Omega _{1}^{\varepsilon })}^{2}+\int_{0}^{t}\left\Vert \nabla 
\boldsymbol{v}_{\varepsilon }(\tau )\right\Vert _{L^{2}(\Omega
_{2}^{\varepsilon })}^{2}d\tau \leq C  \label{2.10}
\end{equation}

Now, from 
\begin{equation*}
\boldsymbol{u}_{\varepsilon }(t)=\int_{0}^{t}\boldsymbol{u}_{\varepsilon
}^{\prime }(\tau )d\tau
\end{equation*}%
we get 
\begin{equation*}
\left\Vert \boldsymbol{u}_{\varepsilon }(t)\right\Vert _{L^{2}(\Omega
_{1}^{\varepsilon })}^{2}\leq \left( \int_{0}^{t}\left\Vert \boldsymbol{u}%
_{\varepsilon }^{\prime }(\tau )\right\Vert _{L^{2}(\Omega _{1}^{\varepsilon
})}d\tau \right) ^{2},
\end{equation*}%
thus, making use of (\ref{2.10}) it holds that 
\begin{equation*}
\left\Vert \boldsymbol{u}_{\varepsilon }(t)\right\Vert _{L^{2}(\Omega
_{1}^{\varepsilon })}^{2}\leq C,
\end{equation*}%
$C>0$ being independent of both $t$ and $\varepsilon $. All the estimates (%
\ref{2.1})-(\ref{2.31}) follow thereby. Finally, for (\ref{2.3}), we proceed
exactly as in the proof of \cite[Section 3.2]{DGHL2} (see also \cite[Lemma
2.3]{JMS}). This completes the proof.
\end{proof}

In order to deal with the compactness result, let us recall the definition
of the auxiliary functions we have just used in the preceding result:%
\begin{equation*}
\boldsymbol{w}_{\varepsilon }(t)=\int_{0}^{t}\boldsymbol{v}_{\varepsilon
}(s)ds,\ \boldsymbol{u}^{\varepsilon }=\chi _{1}^{\varepsilon }\boldsymbol{u}%
_{\varepsilon }+\chi _{2}^{\varepsilon }\boldsymbol{w}_{\varepsilon }\text{
and }\rho ^{\varepsilon }=\chi _{1}^{\varepsilon }\rho _{1}^{\varepsilon
}+\chi _{2}^{\varepsilon }\rho _{2}^{\varepsilon },
\end{equation*}%
With this in mind, we easily see that the vector function $(\boldsymbol{u}%
_{\varepsilon },\boldsymbol{w}_{\varepsilon },p_{\varepsilon })$ satisfies
the weak formulation (\ref{1.1}), (\ref{1.2'})-(\ref{1.4'}) and (\ref{1.5})-(%
\ref{1.7}) where 
\begin{equation}
\rho _{2}^{\varepsilon }\frac{\partial ^{2}\boldsymbol{w}_{\varepsilon }}{%
\partial t^{2}}-\Div\left( B_{0}^{\varepsilon }\nabla \frac{\partial 
\boldsymbol{w}_{\varepsilon }}{\partial t}+\int_{0}^{t}B_{1}^{\varepsilon
}(x,t-s)\nabla \frac{\partial \boldsymbol{w}_{\varepsilon }}{\partial t}%
(x,s)ds\right) +\nabla p_{\varepsilon }=\rho _{2}^{\varepsilon }\boldsymbol{g%
}\text{ in }\Omega _{2,T}^{\varepsilon }  \label{1.2'}
\end{equation}%
\begin{equation}
\Div\frac{\partial \boldsymbol{w}_{\varepsilon }}{\partial t}=0\text{ in }%
\Omega _{2,T}^{\varepsilon }  \label{1.3'}
\end{equation}%
\begin{equation}
\boldsymbol{u}_{\varepsilon }=\boldsymbol{w}_{\varepsilon }\text{ on }\Gamma
_{12}^{\varepsilon }\times (0,T)  \label{1.4'}
\end{equation}%
where $\Omega _{2,T}^{\varepsilon }=\Omega _{2}^{\varepsilon }\times (0,T)$.
In Eq. (\ref{1.2'}), $\boldsymbol{w}_{\varepsilon }$ represents the fluid
displacement. In Section 5 we shall deal with the asymptotic behaviour of
the above reformulation of our problem. It follows at once from Theorem \ref%
{t2.1} that the above system possesses a unique solution $(\boldsymbol{u}%
_{\varepsilon },\boldsymbol{w}_{\varepsilon },p_{\varepsilon })$ such that $%
\boldsymbol{w}_{\varepsilon }\in W^{2,\infty }(0,T;H_{\varepsilon })\cap
H^{1}(0,T;V_{\varepsilon }^{2})$. Moreover it holds that

\begin{proposition}
\label{p2.1}The sequence $(\boldsymbol{u}^{\varepsilon })_{\varepsilon >0}$
is relatively compact in $H^{1}(0,T;L^{2}(\Omega )^{N})$.
\end{proposition}

\begin{proof}
First, in view of interface conditions (\ref{1.4'}) and (\ref{1.5}) we have $%
\nabla \boldsymbol{u}^{\varepsilon }=\chi _{1}^{\varepsilon }\nabla 
\boldsymbol{u}_{\varepsilon }+\chi _{2}^{\varepsilon }\nabla \boldsymbol{w}%
_{\varepsilon }$. Next, the outer boundary conditions (\ref{1.6}) imply $%
\boldsymbol{u}^{\varepsilon }=0$ on $\partial \Omega $. Thus $\boldsymbol{u}%
^{\varepsilon }\in W^{1,\infty }(0,T;H_{0}^{1}(\Omega )^{N})$. Besides, the
inequalities in Theorem \ref{t2.1} ensure that the sequence $(\boldsymbol{u}%
^{\varepsilon })_{\varepsilon >0}$ is bounded in the space 
\begin{equation*}
W=\left\{ \boldsymbol{v}\in L^{2}(0,T;H_{0}^{1}(\Omega )^{N}):\frac{\partial 
\boldsymbol{v}}{\partial t}\in L^{2}(0,T;H_{0}^{1}(\Omega )^{N})\text{ and }%
\frac{\partial ^{2}\boldsymbol{v}}{\partial t^{2}}\in L^{2}(0,T;L^{2}(\Omega
)^{N})\right\} ,
\end{equation*}%
and the latter space is compactly embedded in $H^{1}(0,T;L^{2}(\Omega )^{N})$%
. The result follows thereby.
\end{proof}

\section{Almost periodic functions}

Let $\mathcal{B}(\mathbb{R}^{N})$ denote the Banach space of all bounded
continuous (real-valued) functions on $\mathbb{R}^{N}$ endowed with the $%
\sup $ norm topology.

A mapping $u:\mathbb{R}^{N}\rightarrow \mathbb{R}$ is called an almost
periodic function if $u\in \mathcal{B}(\mathbb{R}^{N})$ and further the set
of all its translates $\{u(\cdot +a)\}_{a\in \mathbb{R}^{N}}$ has a compact
closure in $\mathcal{B}(\mathbb{R}^{N})$. We denote by \textrm{$AP$}$(%
\mathbb{R}^{N})$ the set of all continuous almost periodic functions on $%
\mathbb{R}^{N}$. $\mathrm{AP}(\mathbb{R}^{N})$ is a commutative $\mathcal{C}%
^{\ast }$-algebra with identity. An argument due to Bohr \cite{Bohr} (see
also \cite{Besicovitch}) specifies that a function $u\in \mathrm{AP}(\mathbb{%
R}^{N})$ if and only if $u$ may be uniformly approximated by finite linear
combinations of functions in the set $\{\cos (k\cdot y),\,\sin (k\cdot
y):k\in \mathbb{R}^{N}\}$. It emerges from the above equivalent definition
that $\mathrm{AP}(\mathbb{R}^{N})$ consists of uniformly continuous
functions. Furthermore, it is now well-known that $\mathrm{AP}(\mathbb{R}%
^{N})$ enjoys the following properties:

\begin{itemize}
\item[(P)$_{1}$] $u(\cdot +a)\in \mathrm{AP}(\mathbb{R}^{N})$ whenever $u\in 
\mathrm{AP}(\mathbb{R}^{N})$ and for every $a\in \mathbb{R}^{N}$;

\item[(P)$_{2}$] For each $u\in \mathrm{AP}(\mathbb{R}^{N})$ the closed
convex hull of $\{u(\cdot +a)\}_{a\in \mathbb{R}^{N}}$ in $\mathcal{B}(%
\mathbb{R}^{N})$ contains a unique complex constant $M(u)$ called the mean
value of $u$, and which satisfies the property that the sequence $%
(u^{\varepsilon })_{\varepsilon >0}$ (where $u^{\varepsilon
}(x)=u(x/\varepsilon )$, $x\in \mathbb{R}^{N}$) weakly $\ast $-converges in $%
L^{\infty }(\mathbb{R}^{N})$ to $M(u)$ as $\varepsilon \rightarrow 0$.
Moreover the following properties of $M$ are in order:

\begin{itemize}
\item[(i)] $M$ is nonnegative, i.e., $M(u)\geq 0$ for any $u\in \mathrm{AP}(%
\mathbb{R}^{N})$ with $u\geq 0$;

\item[(ii)] $M$ is continuous on $\mathrm{AP}(\mathbb{R}^{N})$;

\item[(iii)] $M(1)=1$;

\item[(iv)] $M$ is translation invariant, i.e., $M(u(\cdot +a))=M(u)$ for
all $u\in \mathrm{AP}(\mathbb{R}^{N})$ and all $a\in \mathbb{R}^{N}$.
\end{itemize}
\end{itemize}

We infer from the above properties that $\mathrm{AP}(\mathbb{R}^{N})$ is an 
\textit{algebra with mean value} on $\mathbb{R}^{N}$ \cite{Jikov}. We denote
by $\mathcal{K}$ its spectrum and by $\mathcal{G}$ the Gelfand
transformation on \textrm{$AP$}$(\mathbb{R}^{N})$. We recall that $\mathcal{K%
}$ is the set of all nonzero multiplicative linear forms on $\mathrm{AP}(%
\mathbb{R}^{N})$, and $\mathcal{G}$ is the mapping of $\mathrm{AP}(\mathbb{R}%
^{N})$ into $\mathcal{C}(\mathcal{K})$ such that $\mathcal{G}(u)(s)=s(u)$ ($%
s\in \mathcal{K}$, $u\in \mathrm{AP}(\mathbb{R}^{N})$). The image $\mathcal{G%
}(u)$ of $u$ will very often be denoted by $\widehat{u}$. We endow $\mathcal{%
K}$ with the relative weak$\ast $ topology on $(\mathrm{AP}(\mathbb{R}%
^{N}))^{\prime }$ (the topological dual of $\mathrm{AP}(\mathbb{R}^{N})$).
In this topology, $\mathcal{G}$ is an isometric $\ast $-isomorphism of the $%
\mathcal{C}^{\ast }$-algebra $\mathrm{AP}(\mathbb{R}^{N})$ onto the $%
\mathcal{C}^{\ast }$-algebra $\mathcal{C}(\mathcal{K})$. $\mathcal{K}$ being
topologized as above, the mapping $j:\mathbb{R}^{N}\rightarrow \mathcal{K}$
given by $j(y)=\delta _{y}$ (the Dirac mass at $y$) is continuous with dense
range.

In view of the properties (i)-(iii) in (P)$_{2}$, the mean value admits an
integral representation with respect to some Radon measure \cite{Hom1} $%
\beta $ (of total mass $1$) as follows: 
\begin{equation*}
M(u)=\int_{\mathcal{K}}\mathcal{G}(u)d\beta \text{\ \ for all }u\in \mathrm{%
AP}(\mathbb{R}^{N})\text{.}
\end{equation*}

The following important result is worth recalling.

\begin{theorem}
\label{t2.1'}The topological space $\mathcal{K}$ can be provided with a
group operation under which it is a compact topological Abelian group
additively written whose the Haar measure is precisely the Radon measure $%
\beta $. Moreover, the mapping $j:\mathbb{R}^{N}\rightarrow \mathcal{K}$
given by $j(y)=\delta _{y}$ is a continuous group homomorphism satisfying
the property 
\begin{equation}
\widehat{\tau _{y}u}=\tau _{j(y)}\widehat{u}\text{, all }u\in \mathrm{AP}(%
\mathbb{R}^{N})\text{ and all }y\in \mathbb{R}^{N}  \label{2.1'}
\end{equation}%
where $\widehat{\cdot }$ denotes the Gelfand transformation on $\mathrm{AP}(%
\mathbb{R}^{N})$, $\tau _{y}u=u(\cdot +y)$ and $\tau _{j(y)}\widehat{u}=%
\widehat{u}(\cdot +\delta _{y})$.
\end{theorem}

\begin{proof}
The fact that $\mathcal{K}$ can be provided with a group operation is
classically known; see e.g. \cite{DG}. For the proof of (\ref{2.1'}), see
the proof of Proposition 1 in \cite{SW}.
\end{proof}

In the sequel, $\mathcal{K}$ is viewed as a compact Abelian group additively
written, equipped with the normalized Haar measure $\beta $. Hence, as it is
customary, we shall often write $ds$ for $d\beta $.

\begin{remark}
\label{r2.1'}\emph{The group operation defined on }$\mathcal{K}$\emph{\ is
actually the continuous extension of the mapping }$\mathcal{K}\times 
\mathcal{K}\rightarrow \mathcal{K}$\emph{\ defined by }$(\delta _{x},\delta
_{y})\mapsto \delta _{x+y}$\emph{. Since it is additively written, we are
justified in writing }$\delta _{x}+\delta _{y}:=\delta _{x+y}$\emph{\ for }$%
x,y\in \mathbb{R}^{N}$\emph{. Hence, denoting by }$-s$\emph{\ the
symmetrization of }$s\in \mathcal{K}$\emph{, we have in particular }$-\delta
_{y}:=\delta _{-y}$\emph{.}
\end{remark}

Next, we introduce the space $\mathrm{AP}^{\infty }(\mathbb{R}^{N})=\{u\in 
\mathrm{AP}(\mathbb{R}^{N}):D_{y}^{\alpha }u\in \mathrm{AP}(\mathbb{R}^{N})$
for every $\alpha =(\alpha _{1},\ldots ,\alpha _{N})\in \mathbb{N}^{N}\}$
where $D_{y}^{\alpha }=\frac{\partial ^{\left\vert \alpha \right\vert }}{%
\partial y_{1}^{\alpha _{1}}\ldots \partial y_{N}^{\alpha _{N}}}$.

Now, let $B_{\mathrm{AP}}^{p}(\mathbb{R}^{N})$ ($1\leq p<\infty $) denote
the space of Besicovitch almost periodic functions on $\mathbb{R}^{N}$, that
is the closure of $\mathrm{AP}(\mathbb{R}^{N})$ with respect to the
Besicovitch seminorm 
\begin{equation*}
\left\Vert u\right\Vert _{p}=\left( \underset{r\rightarrow +\infty }{\lim
\sup }\frac{1}{\left\vert B_{r}\right\vert }\int_{B_{r}}\left\vert
u(y)\right\vert ^{p}dy\right) ^{1/p}
\end{equation*}%
where $B_{r}$ is the open ball of $\mathbb{R}^{N}$ of radius $r$ centered at
the origin. It is known that $B_{\mathrm{AP}}^{p}(\mathbb{R}^{N})$ is a
complete seminormed vector space verifying $B_{\mathrm{AP}}^{q}(\mathbb{R}%
^{N})\subset B_{\mathrm{AP}}^{p}(\mathbb{R}^{N})$ for $1\leq p\leq q<\infty $%
. Using this last property one may naturally define the space $B_{\mathrm{AP}%
}^{\infty }(\mathbb{R}^{N})$ as follows: 
\begin{equation*}
B_{\mathrm{AP}}^{\infty }(\mathbb{R}^{N})=\{f\in \cap _{1\leq p<\infty }B_{%
\mathrm{AP}}^{p}(\mathbb{R}^{N}):\sup_{1\leq p<\infty }\left\Vert
f\right\Vert _{p}<\infty \}\text{.}\;\;\;\;\;\;\;\;\;
\end{equation*}%
We endow $B_{\mathrm{AP}}^{\infty }(\mathbb{R}^{N})$ with the seminorm $%
\left[ f\right] _{\infty }=\sup_{1\leq p<\infty }\left\Vert f\right\Vert
_{p} $, which makes it a complete seminormed space. We recall that the
spaces $B_{\mathrm{AP}}^{p}(\mathbb{R}^{N})$ ($1\leq p\leq \infty $) are not
Fr\'{e}chet spaces since they are not separated. The following properties
are worth noticing \cite[Subsection 2.2]{CMP} (see in particular Theorem 2.6
therein); see also \cite[Subsection 2.2]{NA}:

\begin{itemize}
\item[(\textbf{1)}] The Gelfand transformation $\mathcal{G}:\mathrm{AP}(%
\mathbb{R}^{N})\rightarrow \mathcal{C}(\mathcal{K})$ extends by continuity
to a unique continuous linear mapping, still denoted by $\mathcal{G}$, of $%
B_{\mathrm{AP}}^{p}(\mathbb{R}^{N})$ into $L^{p}(\mathcal{K})$, which in
turn induces an isometric isomorphism $\mathcal{G}_{1}$, of $B_{\mathrm{AP}%
}^{p}(\mathbb{R}^{N})/\mathcal{N}=\mathcal{B}_{\mathrm{AP}}^{p}(\mathbb{R}%
^{N})$ onto $L^{p}(\mathcal{K})$ (where $\mathcal{N}=\{u\in B_{\mathrm{AP}%
}^{p}(\mathbb{R}^{N}):\mathcal{G}(u)=0\}$). Moreover if $u\in B_{\mathrm{AP}%
}^{p}(\mathbb{R}^{N})\cap L^{\infty }(\mathbb{R}^{N})$ then $\mathcal{G}%
(u)\in L^{\infty }(\mathcal{K})$ and $\left\Vert \mathcal{G}(u)\right\Vert
_{L^{\infty }(\mathcal{K})}\leq \left\Vert u\right\Vert _{L^{\infty }(%
\mathbb{R}^{N})}$.

\item[(\textbf{2)}] The mean value $M$, defined on $\mathrm{AP}(\mathbb{R}%
^{N})$, extends by continuity to a positive continuous linear form (still
denoted by $M$) on $B_{\mathrm{AP}}^{p}(\mathbb{R}^{N})$ satisfying $%
M(u)=\int_{\mathcal{K}}\mathcal{G}(u)d\beta $ and $M(u(\cdot +a))=M(u)$ for
each $u\in B_{\mathrm{AP}}^{p}(\mathbb{R}^{N})$ and all $a\in \mathbb{R}^{N}$%
. Moreover for $u\in B_{\mathrm{AP}}^{p}(\mathbb{R}^{N})$ we have $%
\left\Vert u\right\Vert _{p}=\left[ M(\left\vert u\right\vert ^{p})\right]
^{1/p}$.
\end{itemize}

Spaces of almost periodic functions with values in a Banach space are
defined in a natural way, we refer to \cite{Blot} for details. Keep the
following notations in mind: $\mathrm{AP}(\mathbb{R}^{N};\mathbb{R})=\mathrm{%
AP}(\mathbb{R}^{N})$ and $B_{\mathrm{AP}}^{p}(\mathbb{R}^{N};\mathbb{R})=B_{%
\mathrm{AP}}^{p}(\mathbb{R}^{N})$.

Let $\mathbb{R}_{y,\tau }^{N+1}=\mathbb{R}_{y}^{N}\times \mathbb{R}_{\tau }$
denote the space $\mathbb{R}^{N}\times \mathbb{R}$ with generic variables $%
(y,\tau )$. It holds that $\mathrm{AP}(\mathbb{R}_{y,\tau }^{N+1})=\mathrm{AP%
}(\mathbb{R}_{\tau };\mathrm{AP}(\mathbb{R}_{y}^{N}))$ is the closure in $%
\mathcal{B}(\mathbb{R}_{y,\tau }^{N+1})$ of the tensor product $\mathrm{AP}(%
\mathbb{R}_{y}^{N})\otimes \mathrm{AP}(\mathbb{R}_{\tau })$ \cite[%
Proposition 2.3]{CMP}. To avoid heaviness of notation, we may sometimes set $%
A_{y}=\mathrm{AP}(\mathbb{R}_{y}^{N})$, $A_{\tau }=\mathrm{AP}(\mathbb{R}%
_{\tau })$ and $A=\mathrm{AP}(\mathbb{R}_{y,\tau }^{N+1})$. Correspondingly,
we will denote the mean value on $A_{\xi }$ ($\xi =y,\tau $) by $M_{\xi }$.

Now let $1\leq p<\infty $ and consider the $N$-parameter group of isometries 
$\{T(y):y\in \mathbb{R}^{N}\}$ defined by 
\begin{equation*}
T(y):\mathcal{B}_{\mathrm{AP}}^{p}(\mathbb{R}^{N})\rightarrow \mathcal{B}_{%
\mathrm{AP}}^{p}(\mathbb{R}^{N})\text{,\ }T(y)(u+\mathcal{N})=u(\cdot +y)+%
\mathcal{N}\text{ for }u\in B_{\mathrm{AP}}^{p}(\mathbb{R}^{N}).
\end{equation*}%
Since $\mathrm{AP}(\mathbb{R}^{N})$ consists of uniformly continuous
functions, $\{T(y):y\in \mathbb{R}^{N}\}$ is a strongly continuous group in
the following sense: $T(y)(u+\mathcal{N})\rightarrow u+\mathcal{N}$ in $%
\mathcal{B}_{\mathrm{AP}}^{p}(\mathbb{R}^{N})$ as $\left\vert y\right\vert
\rightarrow 0$. Using the isometric isomorphism $\mathcal{G}_{1}$ we
associated to $\{T(y):y\in \mathbb{R}^{N}\}$ the following $N$-parameter
group $\{\overline{T}(y):y\in \mathbb{R}^{N}\}$ defined by 
\begin{equation*}
\begin{array}{l}
\overline{T}(y):L^{p}(\mathcal{K})\rightarrow L^{p}(\mathcal{K}) \\ 
\overline{T}(y)\mathcal{G}_{1}(u+\mathcal{N})=\mathcal{G}_{1}(T(y)(u+%
\mathcal{N}))=\mathcal{G}_{1}(u(\cdot +y)+\mathcal{N})\text{\ for }u\in B_{%
\mathrm{AP}}^{p}(\mathbb{R}^{N})\text{.}%
\end{array}%
\end{equation*}%
The group $\{\overline{T}(y):y\in \mathbb{R}^{N}\}$ is also strongly
continuous. The infinitesimal generator of $T(y)$ (resp. $\overline{T}(y)$)
along the $i$th coordinate direction, denoted by $\overline{\partial }%
/\partial y_{i}$ (resp. $\partial _{i}$) is defined by 
\begin{equation*}
\begin{array}{l}
\frac{\overline{\partial }u}{\partial y_{i}}=\lim_{s\rightarrow
0}s^{-1}\left( T(se_{i})u-u\right) \text{\ in }\mathcal{B}_{\mathrm{AP}}^{p}(%
\mathbb{R}^{N})\text{ } \\ 
\text{(resp. }\partial _{i}v=\lim_{s\rightarrow 0}s^{-1}\left( \overline{T}%
(se_{i})v-v\right) \text{\ in }L^{p}(\mathcal{K})\text{),}%
\end{array}%
\end{equation*}%
where we have used the same letter $u$ to denote the equivalence class of an
element $u\in B_{\mathrm{AP}}^{p}(\mathbb{R}^{N})$ in $\mathcal{B}_{\mathrm{%
AP}}^{p}(\mathbb{R}^{N})$, $e_{i}=(\delta _{ij})_{1\leq j\leq N}$ ($\delta
_{ij}$ being the Kronecker $\delta $). We collect here below some properties
and spaces attached to $\overline{\partial }/\partial y_{i}$ and $\partial
_{i}$. We refer the reader to e.g. \cite[Subsection 3.1]{JMS} for details.
First, we denote by $\varrho $ the canonical mapping of $B_{\mathrm{AP}}^{p}(%
\mathbb{R}^{N})$ onto $\mathcal{B}_{\mathrm{AP}}^{p}(\mathbb{R}^{N})$, that
is, $\varrho (u)=u+\mathcal{N}$ for $u\in B_{\mathrm{AP}}^{p}(\mathbb{R}%
^{N}) $.

Let $1\leq i\leq N$. Then

\begin{enumerate}
\item If $u\in \mathrm{AP}(\mathbb{R}^{N})$ is such that $\frac{\partial u}{%
\partial y_{i}}\in \mathrm{AP}(\mathbb{R}^{N})$, then 
\begin{equation}
\frac{\overline{\partial }}{\partial y_{i}}(\varrho (u))=\varrho \left( 
\frac{\partial u}{\partial y_{i}}\right) .\ \ \ \ \ \ \ \ \ \ \ \ \ \ \ \ \
\ \ \ \ \ \ \ \ \ \ \ \ \ \ \ \   \label{2.2'}
\end{equation}

\item $\mathcal{G}_{1}\circ \frac{\overline{\partial }}{\partial y_{i}}%
=\partial _{i}\circ \mathcal{G}_{1}$ provided that these mappings make sense.
\end{enumerate}

We define the Sobolev type space $\mathcal{B}_{\mathrm{AP}}^{1,p}(\mathbb{R}%
^{N})$ and its smooth counter-part $\mathcal{D}_{\mathrm{AP}}(\mathbb{R}%
^{N}) $ as follows:%
\begin{equation*}
\mathcal{B}_{\mathrm{AP}}^{1,p}(\mathbb{R}^{N})=\{u\in \mathcal{B}_{\mathrm{%
AP}}^{p}(\mathbb{R}^{N}):\frac{\overline{\partial }u}{\partial y_{i}}\in 
\mathcal{B}_{\mathrm{AP}}^{p}(\mathbb{R}^{N})\ \forall 1\leq i\leq N\}
\end{equation*}%
and 
\begin{equation*}
\mathcal{D}_{\mathrm{AP}}(\mathbb{R}^{N})=\varrho (\mathrm{AP}^{\infty }(%
\mathbb{R}^{N})).
\end{equation*}%
First we have that $\mathcal{D}_{\mathrm{AP}}(\mathbb{R}^{N})$ is dense in $%
\mathcal{B}_{\mathrm{AP}}^{p}(\mathbb{R}^{N})$, $1\leq p<\infty $. Next,
endowed with the norm 
\begin{equation*}
\left\Vert u\right\Vert _{1,p}=\left( \left\Vert u\right\Vert
_{p}^{p}+\sum_{i=1}^{N}\left\Vert \frac{\overline{\partial }u}{\partial y_{i}%
}\right\Vert _{p}^{p}\right) ^{1/p}\ \ (u\in \mathcal{B}_{\mathrm{AP}}^{1,p}(%
\mathbb{R}^{N})),
\end{equation*}%
$\mathcal{B}_{\mathrm{AP}}^{1,p}(\mathbb{R}^{N})$ is a Banach space
admitting $\mathcal{D}_{\mathrm{AP}}(\mathbb{R}^{N})$ as a dense subspace.
We still write $\widehat{u}$ either for $\mathcal{G}(u)$ if $u\in B_{\mathrm{%
AP}}^{p}(\mathbb{R}^{N})$ or for $\mathcal{G}_{1}(u)$ if $u\in \mathcal{B}_{%
\mathrm{AP}}^{p}(\mathbb{R}^{N})$.

We now define the appropriate space of correctors. Prior to that, we set 
\begin{equation*}
\mathcal{B}_{\mathrm{AP}}^{1,p}(\mathbb{R}^{N})/\mathbb{R}=\{u\in \mathcal{B}%
_{\mathrm{AP}}^{1,p}(\mathbb{R}^{N}):M(u)=0\}.
\end{equation*}%
We endow it with the norm 
\begin{equation*}
\left\Vert u\right\Vert _{\#,p}=\left( \sum_{i=1}^{N}\left\Vert \frac{%
\overline{\partial }u}{\partial y_{i}}\right\Vert _{p}^{p}\right) ^{1/p}\ \
(u\in \mathcal{B}_{\mathrm{AP}}^{1,p}(\mathbb{R}^{N})/\mathbb{R}).
\end{equation*}%
Under this norm $\mathcal{B}_{\mathrm{AP}}^{1,p}(\mathbb{R}^{N})/\mathbb{R}$
is unfortunately not complete. We denote by $\mathcal{B}_{\#\mathrm{AP}%
}^{1,p}(\mathbb{R}^{N})$ its completion with respect to the above norm and
by $J$ the canonical embedding of $\mathcal{B}_{\mathrm{AP}}^{1,p}(\mathbb{R}%
^{N})/\mathbb{R}$ into $\mathcal{B}_{\#\mathrm{AP}}^{1,p}(\mathbb{R}^{N})$.
It can be easily checked that $\mathcal{D}_{\mathrm{AP}}(\mathbb{R}^{N})/%
\mathbb{R}=\{u\in \mathcal{D}_{\mathrm{AP}}(\mathbb{R}^{N}):M(u)=0\}$ is
dense in $\mathcal{B}_{\mathrm{AP}}^{1,p}(\mathbb{R}^{N})/\mathbb{R}$. The
following hold true:

\begin{itemize}
\item[(P$_{1}$)] The gradient operator $\overline{\nabla }_{y}=\left( \frac{%
\overline{\partial }}{\partial y_{1}},...,\frac{\overline{\partial }}{%
\partial y_{N}}\right) :\mathcal{B}_{\mathrm{AP}}^{1,p}(\mathbb{R}^{N})/%
\mathbb{R}\rightarrow (\mathcal{B}_{\mathrm{AP}}^{p}(\mathbb{R}^{N}))^{N}$
extends by continuity to a unique mapping denoted by $\widetilde{\nabla }%
_{y}=\left( \frac{\widetilde{\partial }}{\partial y_{i}}\right) _{1\leq
i\leq N}:\mathcal{B}_{\#\mathrm{AP}}^{1,p}(\mathbb{R}^{N})\rightarrow (%
\mathcal{B}_{\mathrm{AP}}^{p}(\mathbb{R}^{N}))^{N}$ with the properties 
\begin{equation*}
\frac{\overline{\partial }}{\partial y_{i}}=\frac{\widetilde{\partial }}{%
\partial y_{i}}\circ J
\end{equation*}%
and 
\begin{equation*}
\left\Vert u\right\Vert _{\#,p}=\left( \sum_{i=1}^{N}\left\Vert \frac{%
\widetilde{\partial }u}{\partial y_{i}}\right\Vert _{p}^{p}\right) ^{1/p}\ \ 
\text{for }u\in \mathcal{B}_{\#\mathrm{AP}}^{1,p}(\mathbb{R}^{N}).
\end{equation*}

\item[(P$_{2}$)] The space $J(\mathcal{D}_{\mathrm{AP}}(\mathbb{R}^{N})/%
\mathbb{C})$ is dense in $\mathcal{B}_{\#\mathrm{AP}}^{1,p}(\mathbb{R}^{N})$.

\item[(P$_{3}$)] The mapping $\widetilde{\nabla }_{y}$ is an isometric
embedding of $\mathcal{B}_{\#\mathrm{AP}}^{1,p}(\mathbb{R}^{N})$ onto a
closed subspace of $(\mathcal{B}_{\mathrm{AP}}^{p}(\mathbb{R}^{N}))^{N}$, so
that $\mathcal{B}_{\#\mathrm{AP}}^{1,p}(\mathbb{R}^{N})$ is a reflexive
Banach space. By duality we define the divergence operator $\overline{\Div}%
_{y}:(\mathcal{B}_{\mathrm{AP}}^{p^{\prime }}(\mathbb{R}^{N}))^{N}%
\rightarrow (\mathcal{B}_{\#\mathrm{AP}}^{1,p}(\mathbb{R}^{N}))^{\prime }$ ($%
p^{\prime }=p/(p-1)$) by 
\begin{equation*}
\left\langle \overline{\Div}_{y}u,v\right\rangle =-\left\langle u,\widetilde{%
\nabla }_{y}v\right\rangle \text{\ for }v\in \mathcal{B}_{\#\mathrm{AP}%
}^{1,p}(\mathbb{R}^{N})\text{ and }u=(u_{i})\in (\mathcal{B}_{\mathrm{AP}%
}^{p^{\prime }}(\mathbb{R}^{N}))^{N}\text{,}
\end{equation*}%
where $\left\langle u,\widetilde{\nabla }_{y}v\right\rangle
=\sum_{i=1}^{N}\int_{\mathcal{K}}\widehat{u}_{i}\partial _{i}\widehat{v}%
d\beta $.
\end{itemize}

We may also define the Laplacian operator on $\mathcal{B}_{\mathrm{AP}%
}^{1,p^{\prime }}(\mathbb{R}^{N})$ (denoted here by $\overline{\Delta }_{y}$%
) as follows: 
\begin{equation*}
\left\langle \overline{\Delta }_{y}w,v\right\rangle =\left\langle \overline{%
\Div}_{y}(\overline{\nabla }_{y}w),v\right\rangle =-\left\langle \overline{%
\nabla }_{y}w,\widetilde{\nabla }_{y}v\right\rangle \text{\ for all }v\in 
\mathcal{B}_{\#\mathrm{AP}}^{1,p}(\mathbb{R}^{N}).
\end{equation*}%
Then it is immediate that $\overline{\Delta }_{y}\varrho (u)=\varrho (\Delta
_{y}u)$ for all $u\in \mathrm{AP}^{\infty }(\mathbb{R}^{N})$ where $\Delta
_{y}$ stands for the usual Laplacian operator on $\mathbb{R}_{y}^{N}$.

Let us finally give a word of warning. Since $J$ is an embedding, this
allows us to view $\mathcal{B}_{\mathrm{AP}}^{1,p}(\mathbb{R}^{N})/\mathbb{R}
$ (and hence $\mathcal{D}_{\mathrm{AP}}(\mathbb{R}^{N})/\mathbb{R}$) as a
dense subspace of $\mathcal{B}_{\#\mathrm{AP}}^{1,p}(\mathbb{R}^{N})$. Thus
we shall henceforth omit $J$ in the notation if it is understood from the
context and there is no risk of confusion. Accordingly, $\widetilde{\nabla }%
_{y}$ and $\frac{\widetilde{\partial }}{\partial y_{i}}$ shall still be
denoted by $\overline{\nabla }_{y}$ and $\frac{\overline{\partial }}{%
\partial y_{i}}$ respectively.

\section{The $\Sigma $-convergence}

The letter $E$ will throughout denote any ordinary sequence $(\varepsilon
_{n})_{n\in \mathbb{N}}$ with $0<\varepsilon _{n}\leq 1$ such that $%
\varepsilon _{n}\rightarrow 0$ as $n\rightarrow \infty $. We shall denote "$%
\varepsilon _{n}\rightarrow 0$ as $n\rightarrow \infty $" merely by "$%
\varepsilon \rightarrow 0$". In what follows, we use the same notation as in
the preceding section. We assume in this subsection that $\Omega $ is any
open subset of $\mathbb{R}^{N}$.

\begin{definition}
\label{d2.1'}\emph{Let }$1\leq p<\infty $\emph{. (1) A sequence }$%
(u_{\varepsilon })_{\varepsilon >0}\subset L^{p}(\Omega )$\emph{\ is said to 
}weakly $\Sigma $-converge\emph{\ in }$L^{p}(\Omega )$\emph{\ to some }$%
u_{0}\in L^{p}(\Omega ;\mathcal{B}_{\mathrm{AP}}^{p}(\mathbb{R}^{N}))$\emph{%
\ if as }$\varepsilon \rightarrow 0$\emph{, we have } 
\begin{equation*}
\int_{\Omega }u_{\varepsilon }(x)f\left( x,\frac{x}{\varepsilon }\right)
dx\rightarrow \iint_{\Omega \times \mathcal{K}}\widehat{u}_{0}(x,s)\widehat{f%
}(x,s)dxds
\end{equation*}%
\emph{for every }$f\in L^{p^{\prime }}(\Omega ;\mathrm{AP}(\mathbb{R}^{N}))$%
\emph{\ (}$1/p^{\prime }=1-1/p$\emph{), where }$\widehat{u}_{0}=\mathcal{G}%
_{1}\circ u_{0}$\emph{\ and }$\widehat{f}=\mathcal{G}\circ f$\emph{. We
express this by writing} $u_{\varepsilon }\rightarrow u_{0}$ in $%
L^{p}(\Omega )$-weak $\Sigma $.

\emph{(2) The sequence }$(u_{\varepsilon })_{\varepsilon >0}\subset
L^{p}(\Omega )$\emph{\ is said to }strongly $\Sigma $-converge\emph{\ in }$%
L^{p}(\Omega )$\emph{\ to some }$u_{0}\in L^{p}(\Omega ;\mathcal{B}_{\mathrm{%
AP}}^{p}(\mathbb{R}^{N}))$\emph{\ if it is weakly }$\Sigma $\emph{%
-convergent towards }$u_{0}$\emph{\ and further satisfies the following
condition: }%
\begin{equation*}
\left\Vert u_{\varepsilon }\right\Vert _{L^{p}(\Omega )}\rightarrow
\left\Vert \widehat{u}_{0}\right\Vert _{L^{p}(\Omega \times \mathcal{K})}.
\end{equation*}%
\emph{We denote this by writing }$u_{\varepsilon }\rightarrow u_{0}$\emph{\
in }$L^{p}(\Omega )$\emph{-strong }$\Sigma $\emph{.}
\end{definition}

The following results are in order; see e.g. \cite[Theorems 3.1 and 3.5]{CMP}
for the proof.

\begin{theorem}
\label{t2.2'}Let $1<p<\infty $. Let $(u_{\varepsilon })_{\varepsilon \in E}$
be a bounded sequence in $L^{p}(\Omega )$. Then there exists a subsequence $%
E^{\prime }$ from $E$ such that the sequence $(u_{\varepsilon
})_{\varepsilon \in E^{\prime }}$ is weakly $\Sigma $-convergent in $%
L^{p}(\Omega )$.
\end{theorem}

\begin{theorem}
\label{t2.3'}Let $1<p<\infty $. Let $(u_{\varepsilon })_{\varepsilon \in E}$
be a bounded sequence in $W^{1,p}(\Omega )$. Then there exist a subsequence $%
E^{\prime }$ of $E$ and a couple of functions 
\begin{equation*}
(u_{0},u_{1})\in W^{1,p}(\Omega )\times L^{p}(\Omega ;\mathcal{B}_{\#\mathrm{%
AP}}^{1,p}(\mathbb{R}^{N}))
\end{equation*}%
such that, as $E^{\prime }\ni \varepsilon \rightarrow 0$, 
\begin{equation*}
u_{\varepsilon }\rightarrow u_{0}\ \text{in }W^{1,p}(\Omega )\text{-weak;}
\end{equation*}%
\begin{equation*}
\frac{\partial u_{\varepsilon }}{\partial x_{i}}\rightarrow \frac{\partial
u_{0}}{\partial x_{i}}+\frac{\overline{\partial }u_{1}}{\partial y_{i}}\text{%
\ in }L^{p}(\Omega )\text{-weak }\Sigma \text{, }1\leq i\leq N.
\end{equation*}
\end{theorem}

The next result deals with the $\Sigma $-convergence of a product of
sequences.

\begin{theorem}[{\protect\cite[Theorem 6]{DPDE}}]
\label{t2.4'}Let $1<p,q<\infty $ and $r\geq 1$ be such that $1/r=1/p+1/q\leq
1$. Assume $(u_{\varepsilon })_{\varepsilon \in E}\subset L^{q}(\Omega )$ is
weakly $\Sigma $-convergent in $L^{q}(\Omega )$ to some $u_{0}\in
L^{q}(\Omega ;\mathcal{B}_{\mathrm{AP}}^{q}(\mathbb{R}^{N}))$, and $%
(v_{\varepsilon })_{\varepsilon \in E}\subset L^{p}(\Omega )$ is strongly $%
\Sigma $-convergent in $L^{p}(\Omega )$ to some $v_{0}\in L^{p}(\Omega ;%
\mathcal{B}_{\mathrm{AP}}^{p}(\mathbb{R}^{N}))$. Then the sequence $%
(u_{\varepsilon }v_{\varepsilon })_{\varepsilon \in E}$ is weakly $\Sigma $%
-convergent in $L^{r}(\Omega )$ to $u_{0}v_{0}$.
\end{theorem}

As a consequence of the above theorem the following holds.

\begin{corollary}
\label{c2.1'}Let $(u_{\varepsilon })_{\varepsilon \in E}\subset L^{p}(\Omega
)$ and $(v_{\varepsilon })_{\varepsilon \in E}\subset L^{p^{\prime }}(\Omega
)\cap L^{\infty }(\Omega )$ ($1<p<\infty $ and $p^{\prime }=p/(p-1)$) be two
sequences such that:

\begin{itemize}
\item[(i)] $u_{\varepsilon }\rightarrow u_{0}$ in $L^{p}(\Omega )$-weak $%
\Sigma $;

\item[(ii)] $v_{\varepsilon }\rightarrow v_{0}$ in $L^{p^{\prime }}(\Omega )$%
-strong $\Sigma $;

\item[(iii)] $(v_{\varepsilon })_{\varepsilon \in E}$ is bounded in $%
L^{\infty }(\Omega )$.
\end{itemize}

\noindent Then $u_{\varepsilon }v_{\varepsilon }\rightarrow u_{0}v_{0}$ in $%
L^{p}(\Omega )$-weak $\Sigma $.
\end{corollary}

We are now able to deal with $\Sigma $-convergence of a convolution product
of sequences. Before we can do that, let us first and foremost define the
convolution of functions defined on $\mathcal{K}$. Let $p,q,m\geq 1$ be real
numbers satisfying $\frac{1}{p}+\frac{1}{q}=1+\frac{1}{m}$. For $u\in L^{p}(%
\mathcal{K})$ and $v\in L^{q}(\mathcal{K})$ we define the convolution
product $u\widehat{\ast }v$ as follows: 
\begin{equation*}
(u\widehat{\ast }v)(s)=\int_{\mathcal{K}}u(r)v(s-r)dr\text{, \ a.e. }s\in 
\mathcal{K}.
\end{equation*}%
Then $\widehat{\ast }$ is well-defined and we have that $u\widehat{\ast }%
v\in L^{m}(\mathcal{K})$ with the following Young inequality: 
\begin{equation*}
\left\Vert u\widehat{\ast }v\right\Vert _{L^{m}(\mathcal{K})}\leq \left\Vert
u\right\Vert _{L^{p}(\mathcal{K})}\left\Vert v\right\Vert _{L^{q}(\mathcal{K}%
)}.
\end{equation*}%
Now let $u\in L^{p}(\mathbb{R}^{N};L^{p}(\mathcal{K}))$ and $v\in L^{q}(%
\mathbb{R}^{N};L^{q}(\mathcal{K}))$. We define the \emph{two-scale
convolution} $u\ast \ast v$ as follows: 
\begin{eqnarray*}
(u\ast \ast v)(x,s) &=&\int_{\mathbb{R}^{N}}\left[ \left( u(t,\cdot )%
\widehat{\ast }v(x-t,\cdot )\right) (s)\right] dt \\
&\equiv &\int_{\mathbb{R}^{N}}\int_{\mathcal{K}}u(t,r)\ast v(x-t,s-r)drdt%
\text{, a.e. }(x,s)\in \mathbb{R}^{N}\times \mathcal{K}.
\end{eqnarray*}%
Then $\ast \ast $ is well defined as an element of $L^{m}(\mathbb{R}%
^{N}\times \mathcal{K})$ and satisfies 
\begin{equation*}
\left\Vert u\ast \ast v\right\Vert _{L^{m}(\mathbb{R}^{N}\times \mathcal{K}%
)}\leq \left\Vert u\right\Vert _{L^{p}(\mathbb{R}^{N}\times \mathcal{K}%
)}\left\Vert v\right\Vert _{L^{q}(\mathbb{R}^{N}\times \mathcal{K})}.
\end{equation*}%
It is to be noted that if $u\in L^{p}(\Omega ;L^{p}(\mathcal{K}))$ where $%
\Omega $ is an open subset of $\mathbb{R}^{N}$, and $v\in L^{q}(\mathbb{R}%
^{N};L^{q}(\mathcal{K}))$, we may still define $u\ast \ast v$ by viewing $u$
as defined in the whole of $\mathbb{R}^{N}\times \mathcal{K}$; it suffices
to take the extension by zero of $u$ outside $\Omega \times \mathcal{K}$.

Finally, for $u\in L^{p}(\mathbb{R}^{N};\mathcal{B}_{\mathrm{AP}}^{p}(%
\mathbb{R}^{N}))$ and $v\in L^{q}(\mathbb{R}^{N};\mathcal{B}_{\mathrm{AP}%
}^{q}(\mathbb{R}^{N}))$ we define the two-scale convolution denoted by $\ast
\ast $ as follows: $u\ast \ast v$ is that element of $L^{m}(\mathbb{R}^{N};%
\mathcal{B}_{\mathrm{AP}}^{m}(\mathbb{R}^{N}))$ defined by 
\begin{equation*}
\mathcal{G}_{1}(u\ast \ast v)=\widehat{u}\ast \ast \widehat{v}.
\end{equation*}%
It also satisfies Young inequality.

Now, let $t\in \mathbb{R}^{N}$ be fixed. Then $(\delta _{t/\varepsilon
})_{\varepsilon >0}$ is a net in the compact topological group $\mathcal{K}$%
, hence it possesses a weak$\ast $ cluster point $r\in \mathcal{K}$. In the
sequel we shall consider a subnet still denoted by $(\delta _{t/\varepsilon
})_{\varepsilon >0}$ (if there is no danger of confusion) that converges to $%
r$ in the topology of $\mathcal{K}$, i.e. 
\begin{equation}
\delta _{\frac{t}{\varepsilon }}\rightarrow r\text{ in }\mathcal{K}\text{%
-weak}\ast \text{ as }\varepsilon \rightarrow 0\text{.}  \label{2.3'}
\end{equation}%
Let $(u_{\varepsilon })_{\varepsilon >0}$ be a sequence in $L^{p}(\Omega )$ (%
$1\leq p<\infty $) which is weakly $\Sigma $-convergent to $u_{0}\in
L^{p}(\Omega ;\mathcal{B}_{\mathrm{AP}}^{p}(\mathbb{R}^{N}))$. We will see
that a macro-translation of $u_{\varepsilon }$ induces both micro- and
macro-translations on the limit $u_{0}$. This is the main goal of the
following result whose proof can be found in \cite{SW}.

\begin{theorem}[{\protect\cite[Proposition 3]{SW}}]
\label{t2.5'}Let $(u_{\varepsilon })_{\varepsilon >0}$ be a sequence in $%
L^{p}(\Omega )$ \emph{(}$1\leq p<\infty $\emph{)} and let the sequence $%
(v_{\varepsilon })_{\varepsilon >0}$ be defined by 
\begin{equation*}
v_{\varepsilon }(x)=u_{\varepsilon }(x+t)\ \ \ (x\in \Omega -t)\text{\ }.
\end{equation*}%
Finally, let $u_{0}\in L^{p}(\Omega ;\mathcal{B}_{\mathrm{AP}}^{p}(\mathbb{R}%
^{N}))$ and assume that $u_{\varepsilon }\rightarrow u_{0}$ in $L^{p}(\Omega
)$-weak $\Sigma $. If \emph{(\ref{2.3'})} holds true then 
\begin{equation}
v_{\varepsilon }\rightarrow v_{0}\text{ in }L^{p}(\Omega -t)\text{-weak }%
\Sigma  \label{2.4'}
\end{equation}%
where $v_{0}\in L^{p}(\Omega -t,\mathcal{B}_{\mathrm{AP}}^{p}(\mathbb{R}%
^{N}))$ is defined by $\widehat{v}_{0}(x,s)=\widehat{u}_{0}(x+t,s+r)$ for $%
(x,s)\in (\Omega -t)\times \mathcal{K}$.
\end{theorem}

We have now in hands all the ingredients to prove the convergence result
involving convolution product of sequences.

Let $p,q,m\geq 1$ be real numbers such that $\frac{1}{p}+\frac{1}{q}=1+\frac{%
1}{m}$. Let $(u_{\varepsilon })_{\varepsilon >0}\subset L^{p}(\Omega )$ and $%
(v_{\varepsilon })_{\varepsilon >0}\subset L^{q}(\mathbb{R}^{N})$ be two
sequences. One may view $u_{\varepsilon }$ as defined in the whole $\mathbb{R%
}^{N}$ by taking its extension by zero outside $\Omega $. Define 
\begin{equation*}
(u_{\varepsilon }\ast v_{\varepsilon })(x)=\int_{\mathbb{R}%
^{N}}u_{\varepsilon }(t)v_{\varepsilon }(x-t)dt\ \ (x\in \mathbb{R}^{N}),
\end{equation*}%
which lies in $L^{m}(\mathbb{R}^{N})$ and satisfies the Young's inequality 
\begin{equation}
\left\Vert u_{\varepsilon }\ast v_{\varepsilon }\right\Vert _{L^{m}(\Omega
)}\leq \left\Vert u_{\varepsilon }\right\Vert _{L^{p}(\Omega )}\left\Vert
v_{\varepsilon }\right\Vert _{L^{q}(\Omega )}.  \label{2.6'}
\end{equation}%
Here is that result.

\begin{theorem}[{\protect\cite[Theorem 2.6]{M2AS}}]
\label{t2.6'}Let $(u_{\varepsilon })_{\varepsilon >0}$ and $(v_{\varepsilon
})_{\varepsilon >0}$ be as above. Assume that, as $\varepsilon \rightarrow 0$%
, $u_{\varepsilon }\rightarrow u_{0}$ in $L^{p}(\Omega )$-weak $\Sigma $ and 
$v_{\varepsilon }\rightarrow v_{0}$ in $L^{q}(\mathbb{R}^{N})$-strong $%
\Sigma $, where $u_{0}$ and $v_{0}$ are in $L^{p}(\Omega ;\mathcal{B}_{%
\mathrm{AP}}^{p}(\mathbb{R}^{N}))$ and $L^{q}(\mathbb{R}^{N};\mathcal{B}_{%
\mathrm{AP}}^{q}(\mathbb{R}^{N}))$ respectively. Then, as $\varepsilon
\rightarrow 0$, 
\begin{equation*}
u_{\varepsilon }\ast v_{\varepsilon }\rightarrow u_{0}\ast \ast v_{0}\text{
in }L^{m}(\Omega )\text{-weak }\Sigma \text{.}
\end{equation*}
\end{theorem}

\begin{remark}
\label{r3'}\emph{In the case where the open set }$\Omega $\emph{\ is bounded
and further }$p=2$\emph{\ and }$q=1$\emph{, we have provided a full proof of
this result in \cite{SW} (see especially the proof of Theorem 2 therein).}
\end{remark}

In this work, we shall use the evolving version of $\Sigma $-convergence
involving space and time variable. Before stating it, however we need some
preliminaries.

Let $A_{y}=\mathrm{AP}(\mathbb{R}_{y}^{N})$ and $A_{\tau }=\mathrm{AP}(%
\mathbb{R}_{\tau })$. We know that $A=\mathrm{AP}(\mathbb{R}_{y,\tau
}^{N+1}) $ is the closure in $\mathcal{B}(\mathbb{R}_{y,\tau }^{N+1})$ of
the tensor product $A_{y}\otimes A_{\tau }$. We denote by $\mathcal{K}_{y}$
(resp. $\mathcal{K}_{\tau }$, $\mathcal{K}$) the spectrum of $A_{y}$ (resp. $%
A_{\tau }$, $A$). The same letter $\mathcal{G}$ will denote the Gelfand
transformation on $A_{y}$, $A_{\tau }$ and $A$, as well. Points in $\mathcal{%
K}_{y}$ (resp. $\mathcal{K}_{\tau }$) are denoted by $s$ (resp. $s_{0}$).
The Haar measure on the compact group $\mathcal{K}_{y}$ (resp. $\mathcal{K}%
_{\tau }$) is denoted by $ds$ (resp. $ds_{0}$). We have $\mathcal{K}=%
\mathcal{K}_{y}\times \mathcal{K}_{\tau }$ (Cartesian product) and the Haar
measure on $\mathcal{K}$ is precisely the product measure $\beta =ds\otimes
ds_{0}$; the last equality follows in an obvious way by the density of $%
A_{y}\otimes A_{\tau }$ in $A$ and by the Fubini's theorem.

This being so, a sequence $(u_{\varepsilon })_{\varepsilon >0}\subset
L^{p}(Q)$\ ($1\leq p<\infty $) is said to weakly $\Sigma $-converge\ in $%
L^{p}(Q)$\ to some $u_{0}\in L^{p}(Q;\mathcal{B}_{\mathrm{AP}}^{p}(\mathbb{R}%
_{y,\tau }^{N+1}))$\ if as $\varepsilon \rightarrow 0$, we have 
\begin{equation*}
\int_{Q}u_{\varepsilon }(x,t)f\left( x,t,\frac{x}{\varepsilon },\frac{t}{%
\varepsilon }\right) dxdt\rightarrow \iint_{Q\times \mathcal{K}}\widehat{u}%
_{0}(x,t,s,s_{0})\widehat{f}(x,t,s,s_{0})dxdtdsds_{0}
\end{equation*}%
for every $f\in L^{p^{\prime }}(Q;\mathrm{AP}(\mathbb{R}_{y,\tau }^{N+1}))$\
($1/p^{\prime }=1-1/p$).

\begin{remark}
\label{r2.2'}\emph{The conclusions of Theorems \ref{t2.2'} and \ref{t2.3'}
are still valid mutatis mutandis in the present context (change there }$%
\Omega $\emph{\ into }$Q$\emph{\ in Theorem \ref{t2.2'}, }$W^{1,p}(\Omega )$%
\emph{\ into }$L^{p}(0,T;W^{1,p}(\Omega ))$\emph{\ and }$W^{1,p}(\Omega
)\times L^{p}(\Omega ;\mathcal{B}_{\#\mathrm{AP}}^{1,p}(\mathbb{R}^{N}))$%
\emph{\ into }$L^{p}(0,T;W^{1,p}(\Omega ))\times L^{p}(Q;\mathcal{B}_{%
\mathrm{AP}}^{p}(\mathbb{R}_{\tau };\mathcal{B}_{\#\mathrm{AP}}^{1,p}(%
\mathbb{R}_{y}^{N})))$\emph{).}
\end{remark}

\medskip

Now, let $\tau \in \mathbb{R}$, and let $\left( u_{\varepsilon }\right)
_{\varepsilon >0}\subset L^{p}\left( Q\right) $\ $(1\leq p<\infty )$\ be a
weakly $\Sigma $-convergent sequence in $L^{p}\left( Q\right) $ to $u_{0}\in
L^{p}(Q;\mathcal{B}_{\mathrm{AP}}^{p}(\mathbb{R}^{N+1}))$. Set 
\begin{equation*}
v_{\varepsilon }(x,t)=u_{\varepsilon }(x,t+\tau )\text{ for }(x,t)\in
Q-(0,\tau )\equiv \Omega \times (-\tau ,T-\tau ).
\end{equation*}%
Then $v_{\varepsilon }\rightarrow v_{0}$ in $L^{p}(Q-(0,\tau ))$-weak $%
\Sigma $ where $v_{0}\in L^{p}(Q-(0,\tau );\mathcal{B}_{\mathrm{AP}}^{p}(%
\mathbb{R}^{N+1}))$ is defined by 
\begin{equation*}
\widehat{v}_{0}(x,t,s,s_{0})=\widehat{u}_{0}(x,t+\tau ,s,s_{0}+r_{0})\text{,
\ }(x,t,s,s_{0})\in \lbrack Q-(0,\tau )]\times \mathcal{K},
\end{equation*}%
the micro-translation $r_{0}$ being a weak$\ast $ cluster point of the
sequence $(\delta _{\tau /\varepsilon })_{\varepsilon >0}$. A similar
conclusion holds for Theorem \ref{t2.6'} mutatis mutandis.

Consider a sequence $(v_{\varepsilon })_{\varepsilon >0}\subset L^{p}(Q)$ ($%
1\leq p<\infty $) which is weakly $\Sigma $-convergent in $L^{p}(Q)$ to some 
$v_{0}\in L^{p}(Q;\mathcal{B}_{\mathrm{AP}}^{p}(\mathbb{R}^{N+1}))$. Define 
\begin{equation*}
w_{\varepsilon }(x,t)=\int_{0}^{t}v_{\varepsilon }(x,\tau )d\tau \text{ a.e. 
}(x,t)\in Q.
\end{equation*}%
Then the following important consequence of Theorem \ref{t2.6'} holds.

\begin{corollary}
\label{c4.2}We have 
\begin{equation*}
w_{\varepsilon }\rightarrow w_{0}\text{ in }L^{p}(Q)\text{-weak }\Sigma 
\text{ as }\varepsilon \rightarrow 0
\end{equation*}%
where $w_{0}\in L^{p}(Q;\mathcal{B}_{\mathrm{AP}}^{p}(\mathbb{R}^{N}))$ is
defined by 
\begin{equation*}
\widehat{w}_{0}(x,t,s)=\int_{0}^{t}\int_{\mathcal{K}_{\tau }}\widehat{v}%
_{0}(x,\tau ,s,r_{0})dr_{0}d\tau ,\ (x,t,s)\in Q\times \mathcal{K}_{y}.
\end{equation*}
\end{corollary}

\begin{proof}
By definition, $w_{\varepsilon }(x,t)=(1\ast v_{\varepsilon })(x,t)$ where
here, $1$ stands for the constant sequence assuming value $1$, and the
convolution is taken with respect to the variable $t$. Theorem \ref{t2.6'}
tells us that $w_{\varepsilon }\rightarrow 1\ast \ast v_{0}$ in $L^{p}(Q)$%
-weak $\Sigma $ as $\varepsilon \rightarrow 0$ where by definition, 
\begin{equation*}
(1\ast \ast \widehat{v}_{0})(x,t,s,s_{0})=\int_{0}^{t}\int_{\mathcal{K}%
_{\tau }}\widehat{v}_{0}(x,\tau ,s,s_{0}-r_{0})dr_{0}d\tau \text{ (}%
(x,t,s,s_{0})\in Q\times \mathcal{K}\text{)}
\end{equation*}%
which does not depend on $s_{0}$ as $\int_{\mathcal{K}_{\tau }}\widehat{v}%
_{0}(x,\tau ,s,s_{0}-r_{0})dr_{0}=\int_{\mathcal{K}_{\tau }}\widehat{v}%
_{0}(x,\tau ,s,r_{0})dr_{0}$; remind that $\mathcal{K}_{\tau }$ is a
topological group. The proof is complete.
\end{proof}

\begin{remark}
\label{r4.2}\emph{It follows immediately that the limit of the sequence }$%
(w_{\varepsilon })_{\varepsilon >0}$\emph{\ defined above does not depend on
the microscopic temporal variable }$\tau $\emph{.}
\end{remark}

\section{Asymptotic behaviour}

Let $\theta $ be the characteristic function of the set $S$ in $\mathbb{Z}%
^{N}$, $S$ being considered in Subsection 1.1. We aim at studying the
asymptotics of the sequence $(\boldsymbol{u}_{\varepsilon },\boldsymbol{v}%
_{\varepsilon },p_{\varepsilon })_{\varepsilon >0}$ defined in Section 2,
under the additional assumption

\begin{itemize}
\item[(A4)] 
\begin{equation}
\theta \text{ is almost periodic (see Subsection 1.1),}  \label{5.1'}
\end{equation}%
\begin{equation}
A_{0},B_{0}\in (B_{\mathrm{AP}}^{2}(\mathbb{R}^{N}))^{N^{2}}\text{, }%
A_{1},B_{1}\in \mathrm{AP}(\mathbb{R}^{N+1})^{N^{2}}\text{ and }\rho _{j}\in 
\mathrm{AP}(\mathbb{R}^{N})\text{ (}j=1,2\text{).}  \label{5.2}
\end{equation}
\end{itemize}

However, for some obvious reasons given in Section 2, we shall rather deal
with the sequence $(\boldsymbol{u}_{\varepsilon },\boldsymbol{w}%
_{\varepsilon },p_{\varepsilon })_{\varepsilon >0}$ satisfying (\ref{1.1}), (%
\ref{1.2'})-(\ref{1.4'}), (\ref{1.5})-(\ref{1.7}).

Throughout this section we do not distinguish between a function $u\in B_{%
\mathrm{AP}}^{p}(\mathbb{R}^{m})$ and its equivalence class $u+\mathcal{N}$
in $\mathcal{B}_{\mathrm{AP}}^{p}(\mathbb{R}^{m})$. This will of course not
lead to any confusion. Here $\Omega $ denotes a bounded open subset of $%
\mathbb{R}^{N}$ with Lipschitz boundary $\partial \Omega $. We still set $%
Q=\Omega \times (0,T)$, and the other notations are as in the preceding
sections.

\subsection{Preliminary results}

To begin with, it is important to recall that $\chi _{j}$ ($j=1,2$) is the
characteristic function of the set $G_{j}$ (defined in Section 1) in $%
\mathbb{R}^{N}$. With this in mind, the following result is the point of
departure of the whole homogenization process.

\begin{lemma}
\label{l5.1'}Under assumption \emph{(\ref{5.1'})} we have 
\begin{equation}
\chi _{j}\in B_{\mathrm{AP}}^{2}(\mathbb{R}^{N})\text{ with }M(\chi _{j})>0%
\text{ for }j=1,2.  \label{5.1}
\end{equation}
\end{lemma}

\begin{proof}
By virtue of \cite[Corollary 3.2]{NG1} (see also \cite[Proposition 4.1 and
Corollary 4.2]{ACAP}) we have $\chi _{1}\in B_{\mathrm{AP}}^{2}(\mathbb{R}%
^{N})$ with $M(\chi _{1})>0$. On the other hand, since $\chi _{1}+\chi
_{2}=1 $ a.e. in $\mathbb{R}^{N}$, it follows at once that $\chi _{2}\in B_{%
\mathrm{AP}}^{2}(\mathbb{R}^{N})$ with $M(\chi _{2})\geq 0$. But by the
definition of the geometry of the domain, we have $M(\chi _{1})<1$, and from
the equality $M(\chi _{1})+M(\chi _{2})=1$, it readily follows that $M(\chi
_{2})>0$.
\end{proof}

The next result relies on the previous one and its easy proof can be found
in \cite{NG1} (see also \cite{ACAP}).

\begin{lemma}
\label{l5.1}Let $(u_{\varepsilon })_{\varepsilon >0}$ be a sequence in $%
L^{p}(Q)$ ($1\leq p<\infty $) which weakly $\Sigma $-converges in $L^{p}(Q)$
towards $u_{0}\in L^{p}(Q;\mathcal{B}_{\mathrm{AP}}^{2}(\mathbb{R}^{N+1}))$.
For $j=1,2,$ let $G_{j}$ be the set defined in Section \emph{1}. Then, as $%
\varepsilon \rightarrow 0$, 
\begin{equation*}
u_{\varepsilon }\chi _{j}^{\varepsilon }\rightarrow u_{0}\chi _{j}\text{ in }%
L^{p}(Q)\text{-weak }\Sigma \text{.}
\end{equation*}
\end{lemma}

Let $E$ be an ordinary sequence defined as at the beginning of Section 4. As
in Section 2, let 
\begin{equation}
\boldsymbol{u}^{\varepsilon }=\chi _{1}^{\varepsilon }\boldsymbol{u}%
_{\varepsilon }+\chi _{2}^{\varepsilon }\boldsymbol{w}_{\varepsilon },
\label{5.3}
\end{equation}%
$\boldsymbol{u}_{\varepsilon }$ and $\boldsymbol{w}_{\varepsilon }$ be as in
Section 2. Then $\boldsymbol{u}^{\varepsilon }$ is the displacement of the $%
\varepsilon $-fluid-solid structure. We know from Section 2 that $%
\boldsymbol{u}^{\varepsilon }\in H^{1}(0,T;H_{0}^{1}(\Omega )^{N})$ with $%
\partial ^{2}\boldsymbol{u}^{\varepsilon }/\partial t^{2}\in
L^{2}(0,T;L^{2}(\Omega )^{N})\equiv L^{2}(Q)^{N}$. Also, as in Section 2, we
set $\rho ^{\varepsilon }=\chi _{1}^{\varepsilon }\rho _{1}^{\varepsilon
}+\chi _{2}^{\varepsilon }\rho _{2}^{\varepsilon }$ (the density of the $%
\varepsilon $-fluid-solid structure). By virtue of Proposition \ref{p2.1},
we derive the existence of a subsequence $E^{\prime }$ of $E$ and of a
function $\boldsymbol{u}_{0}\in H^{1}(0,T;L^{2}(\Omega )^{N})$ such that 
\begin{equation}
\boldsymbol{u}^{\varepsilon }\rightarrow \boldsymbol{u}_{0}\text{ in }%
H^{1}(0,T;L^{2}(\Omega )^{N})\text{ strongly as }E^{\prime }\ni \varepsilon
\rightarrow 0.  \label{5.4}
\end{equation}%
Invoking the diagonal process, we extract from $(\boldsymbol{u}^{\varepsilon
})_{\varepsilon \in E^{\prime }}$ another subsequence (still denoted by $(%
\boldsymbol{u}^{\varepsilon })_{\varepsilon \in E^{\prime }}$) such that 
\begin{equation}
\boldsymbol{u}^{\varepsilon }\rightarrow \boldsymbol{u}_{0}\text{ in }%
H^{1}(0,T;H_{0}^{1}(\Omega )^{N})\text{-weak and }\frac{\partial ^{2}%
\boldsymbol{u}^{\varepsilon }}{\partial t^{2}}\rightarrow \frac{\partial ^{2}%
\boldsymbol{u}_{0}}{\partial t^{2}}\text{ in }L^{2}(Q)^{N}\text{-weak.}
\label{5.5}
\end{equation}%
It follows that $\boldsymbol{u}_{0}\in H^{1}(0,T;H_{0}^{1}(\Omega )^{N})$
with $\partial ^{2}\boldsymbol{u}_{0}/\partial t^{2}\in L^{2}(Q)^{N}$. We
infer from Theorem \ref{t2.3'} and Remark \ref{r2.2'} that (up to a
subsequence of $E^{\prime }$ not relabeled) 
\begin{equation}
\nabla \boldsymbol{u}^{\varepsilon }\rightarrow \nabla \boldsymbol{u}_{0}+%
\overline{\nabla }_{y}\boldsymbol{u}_{1}\text{ in }L^{2}(Q)^{N^{2}}\text{%
-weak }\Sigma  \label{5.5'}
\end{equation}%
where $\boldsymbol{u}_{1}\in L^{2}(Q;\mathcal{B}_{\mathrm{AP}}^{2}(\mathbb{R}%
_{\tau };\mathcal{B}_{\#\mathrm{AP}}^{1,2}(\mathbb{R}_{y}^{N}))^{N})$.
Applying Lemma \ref{l5.1}, we readily see that, as $E^{\prime }\ni
\varepsilon \rightarrow 0$, 
\begin{equation}
\chi _{1}^{\varepsilon }\nabla \boldsymbol{u}_{\varepsilon }\equiv \chi
_{1}^{\varepsilon }\nabla \boldsymbol{u}^{\varepsilon }\rightarrow \chi
_{1}(\nabla \boldsymbol{u}_{0}+\overline{\nabla }_{y}\boldsymbol{u}_{1})%
\text{ in }L^{2}(Q)^{N^{2}}\text{-weak }\Sigma ,  \label{5.6}
\end{equation}%
and 
\begin{equation}
\chi _{2}^{\varepsilon }\nabla \boldsymbol{w}_{\varepsilon }\equiv \chi
_{2}^{\varepsilon }\nabla \boldsymbol{u}^{\varepsilon }\rightarrow \chi
_{2}(\nabla \boldsymbol{u}_{0}+\overline{\nabla }_{y}\boldsymbol{u}_{1})%
\text{ in }L^{2}(Q)^{N^{2}}\text{-weak }\Sigma .  \label{5.7}
\end{equation}%
We also observe that 
\begin{equation}
\chi _{2}^{\varepsilon }\nabla \frac{\partial \boldsymbol{w}_{\varepsilon }}{%
\partial t}\equiv \chi _{2}^{\varepsilon }\nabla \frac{\partial \boldsymbol{u%
}^{\varepsilon }}{\partial t}\rightarrow \chi _{2}(\nabla \frac{\partial 
\boldsymbol{u}_{0}}{\partial t}+\overline{\nabla }_{y}\boldsymbol{v}_{1})%
\text{ in }L^{2}(Q)^{N^{2}}\text{-weak }\Sigma  \label{5.8}
\end{equation}%
when $E^{\prime }\ni \varepsilon \rightarrow 0$, where $M_{\tau }(%
\boldsymbol{v}_{1})=\partial \boldsymbol{u}_{1}/\partial t$ in $G_{2}$.
Indeed for the last convergence result above, we have, in view of (\ref{5.5}%
), 
\begin{equation*}
\chi _{2}^{\varepsilon }\frac{\partial \boldsymbol{w}_{\varepsilon }}{%
\partial t}\rightarrow \chi _{2}\frac{\partial \boldsymbol{u}_{0}}{\partial t%
}\text{ in }L^{2}(Q)^{N^{2}}\text{-weak }\Sigma .
\end{equation*}%
Also, there exists $\boldsymbol{v}_{1}\in L^{2}(Q;\mathcal{B}_{\mathrm{AP}%
}^{2}(\mathbb{R}_{\tau };\mathcal{B}_{\#\mathrm{AP}}^{1,2}(\mathbb{R}%
_{y}^{N}))^{N})$ such that 
\begin{equation*}
\chi _{2}^{\varepsilon }\nabla \frac{\partial \boldsymbol{w}_{\varepsilon }}{%
\partial t}\rightarrow \chi _{2}(\nabla \frac{\partial \boldsymbol{u}_{0}}{%
\partial t}+\overline{\nabla }_{y}\boldsymbol{v}_{1})\text{ in }%
L^{2}(Q)^{N^{2}}\text{-weak }\Sigma .
\end{equation*}%
We seek the relationship between $\boldsymbol{u}_{1}$ and $\boldsymbol{v}%
_{1} $ in $G_{2}$. First, invoking Corollary \ref{c4.2}, 
\begin{equation*}
\chi _{2}^{\varepsilon }\nabla \boldsymbol{w}_{\varepsilon }=\chi
_{2}^{\varepsilon }\int_{0}^{t}\nabla \frac{\partial \boldsymbol{w}%
_{\varepsilon }}{\partial \eta }(\eta )d\eta \rightarrow \chi
_{2}\int_{0}^{t}\left( \nabla \frac{\partial \boldsymbol{u}_{0}}{\partial
\eta }+M_{\tau }(\overline{\nabla }_{y}\boldsymbol{v}_{1})\right) (\eta
)d\eta \text{ in }L^{2}(Q)^{N^{2}}\text{-weak }\Sigma .
\end{equation*}%
The uniqueness of the limit implies 
\begin{equation*}
\chi _{2}(\nabla \boldsymbol{u}_{0}+\overline{\nabla }_{y}\boldsymbol{u}%
_{1})=\chi _{2}\int_{0}^{t}\left( \nabla \frac{\partial \boldsymbol{u}_{0}}{%
\partial \eta }+M_{\tau }(\overline{\nabla }_{y}\boldsymbol{v}_{1})\right)
(\eta )d\eta .
\end{equation*}%
Taking the weak derivative with respect to $t$ in the above equality, 
\begin{equation*}
\chi _{2}\overline{\nabla }_{y}\frac{\partial \boldsymbol{u}_{1}}{\partial t}%
=\chi _{2}M_{\tau }(\overline{\nabla }_{y}\boldsymbol{v}_{1})\text{, or }%
\chi _{2}\overline{\nabla }_{y}\left( \frac{\partial \boldsymbol{u}_{1}}{%
\partial t}-M_{\tau }(\boldsymbol{v}_{1})\right) =0.
\end{equation*}%
The algebra of continuous almost periodic functions being ergodic, it
follows from \cite[Lemma 2.2]{Wright} that $\frac{\partial \boldsymbol{u}_{1}%
}{\partial t}-M_{\tau }(\boldsymbol{v}_{1})$ agrees in $G_{2}$ with a
constant in $\mathbb{R}^{N}$, hence 
\begin{equation}
\frac{\partial \boldsymbol{u}_{1}}{\partial t}-M_{\tau }(\boldsymbol{v}%
_{1})=0\text{\ in }G_{2}  \label{5.8'}
\end{equation}%
since the sole function $\psi $ in $\mathcal{B}_{\mathrm{AP}}^{2}(\mathbb{R}%
_{\tau };\mathcal{B}_{\#\mathrm{AP}}^{1,2}(\mathbb{R}_{y}^{N}))^{N}$
satisfying $\overline{\nabla }_{y}\psi =0$ is zero. It comes that $%
\boldsymbol{u}_{1}\in H^{1}(0,T;L^{2}(\Omega ;\mathcal{B}_{\mathrm{AP}}^{2}(%
\mathbb{R}_{\tau };\mathcal{B}_{\#\mathrm{AP}}^{1,2}(\mathbb{R}%
_{y}^{N}))^{N}))$ does not depend on $\tau $ in $G_{2}$. This will be of
great interest in the next subsection. Indeed, this will allow us to select
test functions not depending on $\tau $ in $G_{2}$.

\subsection{Passage to the limit}

In this section we pass to the limit in the variational formulation of (\ref%
{1.1}), (\ref{1.2'})-(\ref{1.4'}), (\ref{1.5})-(\ref{1.7}) (which is
equivalent to the one of (\ref{1.1})-(\ref{1.7})). For the sake of
simplicity, we shall omit throughout this subsection to precise that $%
E^{\prime }\ni \varepsilon \rightarrow 0$ when dealing with a convergence
result. Bearing this in mind, let $\boldsymbol{\Psi }=(\boldsymbol{\psi }%
_{0},\boldsymbol{\psi }_{1})$ be such that $\boldsymbol{\psi }_{0}=(\psi
_{0}^{k})_{1\leq k\leq N}\in \mathcal{C}_{0}^{\infty }(Q)^{N}$ and $%
\boldsymbol{\psi }_{1}=(\psi _{1,k})_{1\leq k\leq N}\in (\mathcal{C}%
_{0}^{\infty }(Q)\otimes \mathrm{AP}^{\infty }(\mathbb{R}^{N+1}))^{N}$ with $%
\partial \boldsymbol{\psi }_{1}/\partial \tau =0$ in $G_{2}$ (that is, $%
\boldsymbol{\psi }_{1}$ does not depend on $\tau $ in $G_{2}$). We define $%
\boldsymbol{\Psi }_{\varepsilon }=\boldsymbol{\psi }_{0}+\varepsilon 
\boldsymbol{\psi }_{1}^{\varepsilon }$ by 
\begin{equation*}
\boldsymbol{\Psi }_{\varepsilon }(x,t)=\boldsymbol{\psi }_{0}(x,t)+%
\varepsilon \boldsymbol{\psi }_{1}\left( x,t,\frac{x}{\varepsilon },\frac{t}{%
\varepsilon }\right) \text{ for }(x,t)\in Q.
\end{equation*}%
Then we have $\boldsymbol{\Psi }_{\varepsilon }\in \mathcal{C}_{0}^{\infty
}(Q)^{N}$ and, taking $\boldsymbol{\Psi }_{\varepsilon }$ as a test function
in the variational formulation of (\ref{1.1}), (\ref{1.2'})-(\ref{1.4'}), (%
\ref{1.5})-(\ref{1.7}), we get (after using the interface conditions (\ref%
{1.5}) and (\ref{1.4'})) 
\begin{equation}
\begin{array}{l}
-\int_{Q}\rho ^{\varepsilon }\frac{\partial ^{2}\boldsymbol{u}^{\varepsilon }%
}{\partial t^{2}}\cdot \boldsymbol{\Psi }_{\varepsilon }dxdt+\int_{Q}\chi
_{1}^{\varepsilon }\left( A_{0}^{\varepsilon }\nabla \boldsymbol{u}%
_{\varepsilon }+A_{1}^{\varepsilon }\ast \nabla \boldsymbol{u}_{\varepsilon
}\right) \cdot \nabla \boldsymbol{\Psi }_{\varepsilon }dxdt \\ 
\\ 
\ \ \ +\int_{Q}\chi _{2}^{\varepsilon }\left( B_{0}^{\varepsilon }\nabla 
\frac{\partial \boldsymbol{w}_{\varepsilon }}{\partial t}+B_{1}^{\varepsilon
}\ast \nabla \frac{\partial \boldsymbol{w}_{\varepsilon }}{\partial t}%
\right) \cdot \nabla \boldsymbol{\Psi }_{\varepsilon }dxdt \\ 
\\ 
\ \ \ \ -\int_{Q}\chi _{2}^{\varepsilon }p_{\varepsilon }\Div\boldsymbol{%
\Psi }_{\varepsilon }dxdt=\int_{Q}(\chi _{1}^{\varepsilon }\rho
_{1}^{\varepsilon }\boldsymbol{f}+\chi _{2}^{\varepsilon }\rho
_{2}^{\varepsilon }\boldsymbol{g}_{\varepsilon })\cdot \boldsymbol{\Psi }%
_{\varepsilon }dxdt.%
\end{array}
\label{5.9}
\end{equation}%
Our main objective is to pass to the limit in (\ref{5.9}). Let us consider
the terms in (\ref{5.9}) separately. We first deal with the term $%
\int_{Q}\rho ^{\varepsilon }\frac{\partial ^{2}\boldsymbol{u}^{\varepsilon }%
}{\partial t^{2}}\cdot \boldsymbol{\Psi }_{\varepsilon }dxdt$. We have 
\begin{eqnarray*}
\int_{Q}\rho ^{\varepsilon }\frac{\partial ^{2}\boldsymbol{u}^{\varepsilon }%
}{\partial t^{2}}\cdot \boldsymbol{\Psi }_{\varepsilon }dxdt
&=&-\int_{Q}\rho ^{\varepsilon }\frac{\partial \boldsymbol{u}^{\varepsilon }%
}{\partial t}\cdot \frac{\partial \boldsymbol{\Psi }_{\varepsilon }}{%
\partial t}dxdt \\
&=&-\int_{Q}\rho ^{\varepsilon }\frac{\partial \boldsymbol{u}^{\varepsilon }%
}{\partial t}\cdot \left( \frac{\partial \boldsymbol{\psi }_{0}}{\partial t}%
+\left( \frac{\partial \boldsymbol{\psi }_{1}}{\partial \tau }\right)
^{\varepsilon }+\varepsilon \left( \frac{\partial \boldsymbol{\psi }_{1}}{%
\partial t}\right) ^{\varepsilon }\right) dxdt
\end{eqnarray*}%
and 
\begin{equation*}
\frac{\partial \boldsymbol{\Psi }_{\varepsilon }}{\partial t}\rightarrow 
\frac{\partial \boldsymbol{\psi }_{0}}{\partial t}+\frac{\partial 
\boldsymbol{\psi }_{1}}{\partial \tau }\text{ in }L^{2}(Q)^{N}\text{-strong }%
\Sigma .
\end{equation*}%
Hence 
\begin{equation*}
\int_{Q}\rho ^{\varepsilon }\frac{\partial ^{2}\boldsymbol{u}^{\varepsilon }%
}{\partial t^{2}}\cdot \boldsymbol{\Psi }_{\varepsilon }dxdt\rightarrow
\iint_{Q\times \mathcal{K}}\widehat{\rho }\frac{\partial \boldsymbol{u}_{0}}{%
\partial t}\cdot \left( \frac{\partial \boldsymbol{\psi }_{0}}{\partial t}+%
\widehat{\frac{\partial \boldsymbol{\psi }_{1}}{\partial \tau }}\right)
dxdtd\beta
\end{equation*}%
where $\widehat{\rho }=\widehat{\chi }_{1}\widehat{\rho }_{1}+\widehat{\chi }%
_{2}\widehat{\rho }_{2}$. But $\int_{\mathcal{K}}\widehat{\rho }\frac{%
\partial \boldsymbol{u}_{0}}{\partial t}\cdot \widehat{\frac{\partial 
\boldsymbol{\psi }_{1}}{\partial \tau }}d\beta =0$ (as $\boldsymbol{u}_{0}$
does not depend on $\tau $, the microscopic temporal variable) and $%
\iint_{Q\times \mathcal{K}}\widehat{\rho }\frac{\partial \boldsymbol{u}_{0}}{%
\partial t}\cdot \frac{\partial \boldsymbol{\psi }_{0}}{\partial t}%
dxdtd\beta =\left( \int_{\mathcal{K}_{y}}\widehat{\rho }ds\right) \int_{Q}%
\frac{\partial \boldsymbol{u}_{0}}{\partial t}\cdot \frac{\partial 
\boldsymbol{\psi }_{0}}{\partial t}dxdt=-\left( \int_{\mathcal{K}_{y}}%
\widehat{\rho }ds\right) \int_{Q}\frac{\partial ^{2}\boldsymbol{u}_{0}}{%
\partial t^{2}}\cdot \boldsymbol{\psi }_{0}dxdt$. Hence, on letting $\rho
_{0}=\int_{\mathcal{K}_{y}}\widehat{\rho }ds$, we have, 
\begin{equation*}
\int_{Q}\rho ^{\varepsilon }\frac{\partial ^{2}\boldsymbol{u}^{\varepsilon }%
}{\partial t^{2}}\cdot \boldsymbol{\Psi }_{\varepsilon }dxdt\rightarrow \rho
_{0}\int_{Q}\frac{\partial ^{2}\boldsymbol{u}_{0}}{\partial t^{2}}\cdot 
\boldsymbol{\psi }_{0}dxdt.
\end{equation*}%
Now, dealing with the second term of the left-hand side of (\ref{5.9}), one
easily shows that 
\begin{equation}
\nabla \boldsymbol{\Psi }_{\varepsilon }\rightarrow \nabla \boldsymbol{\psi }%
_{0}+\nabla _{y}\boldsymbol{\psi }_{1}\text{ in }L^{2}(Q)^{N^{2}}\text{%
-strong }\Sigma .  \label{5.12}
\end{equation}%
Putting (\ref{5.12}) and (\ref{5.6}) together, we get by Corollary \ref%
{c2.1'} that 
\begin{equation*}
\chi _{1}^{\varepsilon }\nabla \boldsymbol{u}_{\varepsilon }\cdot \nabla 
\boldsymbol{\Psi }_{\varepsilon }\rightarrow \chi _{1}\left( \nabla 
\boldsymbol{u}_{0}+\nabla _{y}\boldsymbol{u}_{1}\right) \cdot \left( \nabla 
\boldsymbol{\psi }_{0}+\nabla _{y}\boldsymbol{\psi }_{1}\right) \text{ in }%
L^{2}(Q)\text{-weak }\Sigma \text{.}
\end{equation*}%
Hence, using $A_{0}$ as test function (recall that $A_{0}\in ((B_{\mathrm{AP}%
}^{2}(\mathbb{R}^{N})\cap L^{\infty }(\mathbb{R}^{N}))^{N^{2}}$ so that by 
\cite[Proposition 8]{DPDE}, it is an admissible test function in the sense
of \cite[Definition 5]{DPDE}) leads to 
\begin{equation*}
\int_{Q}\chi _{1}^{\varepsilon }A_{0}^{\varepsilon }\nabla \boldsymbol{u}%
_{\varepsilon }\cdot \nabla \boldsymbol{\Psi }_{\varepsilon }dxdt\rightarrow
\iint_{Q\times \mathcal{K}}\widehat{A}_{0}\mathbb{D}\boldsymbol{u}\cdot 
\mathbb{D}\boldsymbol{\Psi }dxdtd\beta
\end{equation*}%
where, setting $\boldsymbol{u}=(\boldsymbol{u}_{0},\boldsymbol{u}_{1})$, we
define $\mathbb{D}\boldsymbol{u}=(\mathbb{D}_{j}\boldsymbol{u})_{1\leq j\leq
N}$ with $\mathbb{D}_{j}\boldsymbol{u}=(\mathbb{D}_{j}\boldsymbol{u}%
^{k})_{1\leq k\leq N}$ and $\mathbb{D}_{j}\boldsymbol{u}^{k}=\frac{\partial
u_{0}^{k}}{\partial x_{j}}+\partial _{j}\widehat{u}_{1}^{k}$ ($\partial _{j}%
\widehat{u}_{1}^{k}=\mathcal{G}_{1}\left( \overline{\partial }%
u_{1}^{k}/\partial y_{j}\right) $) with $\boldsymbol{u}_{0}=(u_{0}^{k})_{1%
\leq k\leq N}$ (and the same definition for $\boldsymbol{u}_{1}$). Still for
the same term, let us now deal with the part involving convolution. First we
know that $A_{1}^{\varepsilon }\rightarrow A_{1}$ in $L^{p}(Q)^{N^{2}}$%
-strong $\Sigma $ ($1\leq p<\infty $). Hence by virtue of Theorem \ref{t2.6'}%
, we obtain 
\begin{equation*}
A_{1}^{\varepsilon }\ast \chi _{1}^{\varepsilon }\nabla \boldsymbol{u}%
_{\varepsilon }\rightarrow A_{1}\ast \ast \chi _{1}\left( \nabla \boldsymbol{%
u}_{0}+\nabla _{y}\boldsymbol{u}_{1}\right) \text{ in }L^{2}(Q)^{N^{2}}\text{%
-weak }\Sigma
\end{equation*}%
(just take $p=1$ above). Therefore by repeating the reasoning above, we
quickly arrive at 
\begin{equation*}
\int_{Q}\chi _{1}^{\varepsilon }(A_{1}^{\varepsilon }\ast \nabla \boldsymbol{%
u}_{\varepsilon })\cdot \nabla \boldsymbol{\Psi }_{\varepsilon
}dxdt\rightarrow \iint_{Q\times \mathcal{K}}\widehat{\chi }_{1}(\widehat{A}%
_{1}\ast \ast \mathbb{D}\boldsymbol{u})\cdot \mathbb{D}\boldsymbol{\Psi }%
dxdtd\beta .
\end{equation*}%
Thus 
\begin{equation*}
\int_{Q}\chi _{1}^{\varepsilon }\left( A_{0}^{\varepsilon }\nabla 
\boldsymbol{u}_{\varepsilon }+A_{1}^{\varepsilon }\ast \nabla \boldsymbol{u}%
_{\varepsilon }\right) \cdot \nabla \boldsymbol{\Psi }_{\varepsilon
}dxdt\rightarrow \iint_{Q\times \mathcal{K}}\widehat{\chi }_{1}(\widehat{A}%
_{0}\mathbb{D}\boldsymbol{u}+\widehat{A}_{1}\ast \ast \mathbb{D}\boldsymbol{%
u)}\cdot \mathbb{D}\boldsymbol{\Psi }dxdtd\beta .
\end{equation*}%
Arguing as for the preceding term and accounting of (\ref{5.8'}) and of the
fact that $\boldsymbol{\psi }_{1}$ does not depend on $\tau $ in $G_{2}$, we
get 
\begin{equation*}
\int_{Q}\chi _{2}^{\varepsilon }\left( B_{0}^{\varepsilon }\nabla \frac{%
\partial \boldsymbol{w}_{\varepsilon }}{\partial t}+B_{1}^{\varepsilon }\ast
\nabla \frac{\partial \boldsymbol{w}_{\varepsilon }}{\partial t}\right)
\cdot \nabla \boldsymbol{\Psi }_{\varepsilon }dxdt\rightarrow \iint_{Q\times 
\mathcal{K}}\widehat{\chi }_{2}(\widehat{B}_{0}\frac{\partial }{\partial t}%
\mathbb{D}\boldsymbol{u}+\widehat{B}_{1}\ast \ast \frac{\partial }{\partial t%
}\mathbb{D}\boldsymbol{u)}\cdot \mathbb{D}\boldsymbol{\Psi }dxdtd\beta .
\end{equation*}

As regards the terms with the pressure, one has 
\begin{eqnarray}
\int_{Q}\chi _{2}^{\varepsilon }p_{\varepsilon }\Div\boldsymbol{\Psi }%
_{\varepsilon }dxdt &=&\int_{Q}\chi _{2}^{\varepsilon }p_{\varepsilon }\Div%
\boldsymbol{\psi }_{0}dxdt+\int_{Q}\chi _{2}^{\varepsilon }p_{\varepsilon }(%
\Div_{y}\boldsymbol{\psi }_{1})^{\varepsilon }dxdt  \label{5.12'} \\
&&\ \ +\varepsilon \int_{Q}\chi _{2}^{\varepsilon }p_{\varepsilon }(\Div%
\boldsymbol{\psi }_{1})^{\varepsilon }dxdt.  \notag
\end{eqnarray}%
Owing to (\ref{2.3}), we consider $p\in L^{2}(Q;\mathcal{B}_{\mathrm{AP}%
}^{2}(\mathbb{R}^{N+1}))$ such that 
\begin{equation*}
\chi _{2}^{\varepsilon }p_{\varepsilon }\rightarrow \chi _{2}p\text{ in }%
L^{2}(Q)\text{-weak }\Sigma
\end{equation*}%
up to a subsequence of $E^{\prime }$ not relabeled. Set $p_{0}(x,t)=\int_{%
\mathcal{K}}\widehat{\chi }_{2}\widehat{p}(x,t,s,s_{0})d\beta $ for a.e. $%
(x,t)\in Q$. Then passing to the limit in (\ref{5.12'}) yields 
\begin{equation*}
\int_{Q}\chi _{2}^{\varepsilon }p_{\varepsilon }\Div\boldsymbol{\Psi }%
_{\varepsilon }dxdt\rightarrow \int_{Q}p_{0}\Div\boldsymbol{\psi }%
_{0}dxdt+\iint_{Q\times \mathcal{K}}\widehat{\chi }_{2}\widehat{p}\widehat{%
\Div}\widehat{\boldsymbol{\phi }}_{1}dxdtd\beta .
\end{equation*}%
It also holds that 
\begin{equation*}
\int_{Q}(\chi _{1}^{\varepsilon }\rho _{1}^{\varepsilon }\boldsymbol{f}+\chi
_{2}^{\varepsilon }\rho _{2}^{\varepsilon }\boldsymbol{g})\cdot \boldsymbol{%
\Psi }_{\varepsilon }dxdt\rightarrow \iint_{Q\times \mathcal{K}}(\widehat{%
\chi }_{1}\widehat{\rho }_{1}\boldsymbol{f}+\widehat{\chi }_{2}\widehat{\rho 
}_{2}\boldsymbol{g})\cdot \boldsymbol{\psi }_{0}dxdtd\beta .
\end{equation*}%
Finally, from the equation $\Div\frac{\partial \boldsymbol{w}_{\varepsilon }%
}{\partial t}=0$ in $\Omega _{2}^{\varepsilon }\times (0,T)$ which is
equivalent to $\chi _{2}^{\varepsilon }\Div\frac{\partial \boldsymbol{w}%
_{\varepsilon }}{\partial t}=0$ in $Q$, we get 
\begin{equation}
\iint_{Q\times \mathcal{K}_{y}}\widehat{\chi }_{2}\left( \Div\frac{\partial 
\boldsymbol{u}_{0}}{\partial t}+\widehat{\Div}\widehat{\frac{\partial 
\boldsymbol{u}_{1}}{\partial t}}\right) \widehat{\psi }dxdtds=0.
\label{5.13'}
\end{equation}%
Indeed, since $\boldsymbol{u}_{1}$ does not depend on $\tau $ in $G_{2}$, we
choose a test function $\psi \in \mathcal{C}_{0}^{\infty }(Q)\otimes \mathrm{%
AP}^{\infty }(\mathbb{R}_{y}^{N})$, and we have 
\begin{equation*}
\int_{Q}\chi _{2}^{\varepsilon }\psi ^{\varepsilon }\Div\frac{\partial 
\boldsymbol{w}_{\varepsilon }}{\partial t}dxdt=0.
\end{equation*}%
Passing to the limit above and using (\ref{5.8'}), we are led to (\ref{5.13'}%
).

We have just verified that the triplet $(\boldsymbol{u}_{0},\boldsymbol{u}%
_{1},p)$ determined earlier solves the variational problem 
\begin{equation}
\left\{ 
\begin{array}{l}
\rho _{0}\int_{Q}\frac{\partial ^{2}\boldsymbol{u}_{0}}{\partial t^{2}}\cdot 
\boldsymbol{\psi }_{0}dxdt+\iint_{Q\times \mathcal{K}}\widehat{\chi }_{1}(%
\widehat{A}_{0}\mathbb{D}\boldsymbol{u}+\widehat{A}_{1}\ast \ast \mathbb{D}%
\boldsymbol{u})\cdot \mathbb{D}\boldsymbol{\Psi }dxdtd\beta \\ 
\ \ +\iint_{Q\times \mathcal{K}}\widehat{\chi }_{2}(\widehat{B}_{0}\frac{%
\partial }{\partial t}\mathbb{D}\boldsymbol{u}+\widehat{B}_{1}\ast \ast 
\frac{\partial }{\partial t}\mathbb{D}\boldsymbol{u})\cdot \mathbb{D}%
\boldsymbol{\Psi }dxdtd\beta \\ 
\ \ \ -\int_{Q}p_{0}\Div\boldsymbol{\psi }_{0}dxdt-\iint_{Q\times \mathcal{K}%
}\widehat{\chi }_{2}\widehat{p}\widehat{\Div}\widehat{\boldsymbol{\psi }}%
_{1}dxdtd\beta =\int_{Q}\boldsymbol{F}\cdot \boldsymbol{\psi }_{0}dxdt, \\ 
\iint_{Q\times \mathcal{K}_{y}}\widehat{\chi }_{2}\left( \Div\frac{\partial 
\boldsymbol{u}_{0}}{\partial t}+\widehat{\Div}\widehat{\frac{\partial 
\boldsymbol{u}_{1}}{\partial t}}\right) \widehat{\psi }dxdtds=0 \\ 
\text{for all }(\boldsymbol{\psi }_{0},\boldsymbol{\psi }_{1})\in \mathcal{C}%
_{0}^{\infty }(Q)^{N}\times \lbrack \mathcal{C}_{0}^{\infty }(Q)\otimes 
\mathrm{AP}^{\infty }(\mathbb{R}^{N+1})]^{N}\text{ and }\psi \in \\ 
\mathcal{C}_{0}^{\infty }(Q)\otimes \mathrm{AP}^{\infty }(\mathbb{R}_{y}^{N})%
\end{array}%
\right.  \label{5.13}
\end{equation}%
where 
\begin{equation*}
\rho _{0}=\int_{\mathcal{K}_{y}}(\widehat{\chi }_{1}\widehat{\rho }_{1}+%
\widehat{\chi }_{2}\widehat{\rho }_{2})ds\text{ and }\boldsymbol{F}=\int_{%
\mathcal{K}_{y}}(\widehat{\chi }_{1}\widehat{\rho }_{1}\boldsymbol{f}+%
\widehat{\chi }_{2}\widehat{\rho }_{2}\boldsymbol{g})ds.
\end{equation*}

Taking successively $(\boldsymbol{\psi }_{0},\boldsymbol{\psi }_{1})=(%
\boldsymbol{\psi }_{0},0)$ and $(\boldsymbol{\psi }_{0},\boldsymbol{\psi }%
_{1})=(0,\boldsymbol{\psi }_{1})$ in (\ref{5.13}), we get the following
equivalent system: 
\begin{equation}
\rho _{0}\frac{\partial ^{2}\boldsymbol{u}_{0}}{\partial t^{2}}-\Div\left(
\chi _{1}(A_{0}\overline{\mathbb{D}}_{y}\boldsymbol{u}+A_{1}\ast \ast 
\overline{\mathbb{D}}_{y}\boldsymbol{u})+\chi _{2}\left( B_{0}\frac{\partial 
}{\partial t}\overline{\mathbb{D}}_{y}\boldsymbol{u}+B_{1}\ast \ast \frac{%
\partial }{\partial t}\overline{\mathbb{D}}_{y}\boldsymbol{u}\right) \right)
+\nabla p_{0}=\boldsymbol{F},  \label{5.14'}
\end{equation}%
\begin{equation}
-\overline{\Div}_{y}\left( \chi _{1}(A_{0}\overline{\mathbb{D}}_{y}%
\boldsymbol{u}+A_{1}\ast \ast \overline{\mathbb{D}}_{y}\boldsymbol{u})+\chi
_{2}\left( B_{0}\frac{\partial }{\partial t}\overline{\mathbb{D}}_{y}%
\boldsymbol{u}+B_{1}\ast \ast \frac{\partial }{\partial t}\overline{\mathbb{D%
}}_{y}\boldsymbol{u}\right) \right) +\overline{\nabla }_{y}(\chi _{2}p)=0,
\label{5.15'}
\end{equation}%
\begin{equation}
\chi _{2}\left( \Div\frac{\partial \boldsymbol{u}_{0}}{\partial t}+\overline{%
\Div}_{y}\frac{\partial \boldsymbol{u}_{1}}{\partial t}\right) =0,
\label{5.16'}
\end{equation}%
\begin{equation}
\boldsymbol{u}_{0}(x,0)=\frac{\partial \boldsymbol{u}_{0}}{\partial t}(x,0)=0
\label{5.17'}
\end{equation}%
where here, $\overline{\mathbb{D}}_{y}\boldsymbol{u}=\mathcal{G}_{1}^{-1}(%
\mathbb{D}\boldsymbol{u})=\nabla \boldsymbol{u}_{0}+\overline{\nabla }_{y}%
\boldsymbol{u}_{1}$.

\begin{lemma}
\label{l5.2}The system \emph{(\ref{5.14'})-(\ref{5.17'})} has a unique
solution.
\end{lemma}

\begin{proof}
The proof follows the same lines of reasoning as the proof of \cite[Lemma 5]%
{GM}.
\end{proof}

Our aim now is to derive the effective equations for $\boldsymbol{u}_{0}$.
In order to do so, we need to consider firstly (\ref{5.15'}) and (\ref{5.16'}%
) which will help us to express $\boldsymbol{u}_{1}$ and $p$ in terms of $%
\boldsymbol{u}_{0}$. We shall seek $(\boldsymbol{u}_{1},p)$ under the form $%
\boldsymbol{u}_{1}=\boldsymbol{u}_{11}+\boldsymbol{u}_{12}$ where $%
\boldsymbol{u}_{11}$\ and $(\boldsymbol{u}_{12},p)$ solve respectively Eq. (%
\ref{5.18'}) and (\ref{5.19'}) below: 
\begin{equation}
-\overline{\Div}_{y}\left( \chi _{1}(A_{0}\overline{\mathbb{D}}_{y}%
\boldsymbol{u}+A_{1}\ast \ast \overline{\mathbb{D}}_{y}\boldsymbol{u}%
)\right) =0  \label{5.18'}
\end{equation}%
\begin{equation}
\left\{ 
\begin{array}{l}
-\overline{\Div}_{y}\left( \chi _{2}(B_{0}\frac{\partial }{\partial t}%
\overline{\mathbb{D}}_{y}\boldsymbol{u}+B_{1}\ast \ast \frac{\partial }{%
\partial t}\overline{\mathbb{D}}_{y}\boldsymbol{u})\right) +\overline{\nabla 
}_{y}(\chi _{2}p)=0 \\ 
\chi _{2}\left( \Div\frac{\partial \boldsymbol{u}_{0}}{\partial t}+\overline{%
\Div}_{y}\frac{\partial \boldsymbol{u}_{1}}{\partial t}\right) =0.%
\end{array}%
\right.  \label{5.19'}
\end{equation}

Eq. (\ref{5.18'}) allows us to find $\boldsymbol{u}_{1}$ in $G_{1}$ while (%
\ref{5.19'}) permits us to determine $\boldsymbol{u}_{1}$ in $G_{2}$ and $p$
in $G_{2}$, so that $\boldsymbol{u}_{11}=\chi _{1}\boldsymbol{u}_{1}$ and $%
\boldsymbol{u}_{12}=\chi _{2}\boldsymbol{u}_{1}$. Let us first deal with (%
\ref{5.18'}). For $\xi =(\xi _{ij})_{1\leq i,j\leq N}\in \mathbb{R}^{N^{2}}$
be fixed, consider the cell equation 
\begin{equation}
\left\{ 
\begin{array}{l}
\text{Find }u(\xi )\in \mathcal{B}_{\mathrm{AP}}^{2}(\mathbb{R}_{\tau };%
\mathcal{B}_{\#\mathrm{AP}}^{1,2}(\mathbb{R}_{y}^{N}))^{N}\text{ such that}
\\ 
-\overline{\Div}_{y}(A_{0}(\xi +\overline{\nabla }_{y}u(\xi ))+A_{1}\ast
\ast (\xi +\overline{\nabla }_{y}u(\xi )))=0\text{ in }G_{1} \\ 
u(\xi )=0\text{ in }G_{2}.%
\end{array}%
\right.  \label{5.20'}
\end{equation}%
In view of the assumptions on $A_{0}$ and $A_{1}$, we can verify, using \cite%
[Theorem 6.4]{Orlik}, that (\ref{5.20'}) possesses a unique solution. Taking
in (\ref{5.20'}) the special $\xi =\nabla \boldsymbol{u}_{0}(x,t)$ ($%
(x,t)\in Q$) and comparing the resulting equation with (\ref{5.18'}), we get
by the uniqueness of $u(\xi )$ that $\boldsymbol{u}_{1}(x,t)=u(\nabla 
\boldsymbol{u}_{0}(x,t))$ in $G_{1}$, which define a mapping $u(\nabla 
\boldsymbol{u}_{0}):(x,t)\mapsto u(\nabla \boldsymbol{u}_{0}(x,t))$ from $Q$
into $\mathcal{B}_{\mathrm{AP}}^{2}(\mathbb{R}_{\tau };\mathcal{B}_{\#%
\mathrm{AP}}^{1,2}(\mathbb{R}_{y}^{N}))^{N}$.

Now, as for (\ref{5.19'}), fix again $\xi =(\xi _{ij})_{1\leq i,j\leq N}$ in 
$\mathbb{R}^{N^{2}}$ and consider this time the following equation 
\begin{equation}
\left\{ 
\begin{array}{l}
-\overline{\Div}_{y}\left( B_{0}(\xi +\overline{\nabla }_{y}v(\xi
))+B_{1}\ast \ast (\xi +\overline{\nabla }_{y}v(\xi ))\right) +\overline{%
\nabla }_{y}\pi (\xi )=0\text{ in }G_{2} \\ 
\mathrm{tr}(\xi )+\overline{\Div}_{y}v(\xi )=0\text{ in }G_{2}\text{ and }%
v(\xi )=0\text{ in }G_{1}%
\end{array}%
\right.  \label{5.21'}
\end{equation}%
where $\mathrm{tr}(\xi )=\sum_{i=1}^{N}\xi _{ii}$ is the trace of $\xi $.
Let $\mathbb{B}_{\mathrm{AP}}^{1,2}(G_{2})$ denote the strong closure in $%
\mathcal{B}_{\#\mathrm{AP}}^{1,2}(\mathbb{R}^{N})^{N}$ of the set $\mathcal{V%
}_{2}=\{\psi \in (\mathcal{D}_{A}(\mathbb{R}^{N}))^{N}:M_{y}(\psi )=0$ and $%
\psi =0$ in $G_{1}\}$. Then following the standard methods for solving
compressible Stokes equations \cite[Chap. 1]{Temam}, it also follows from 
\cite[Theorem 6.4]{Orlik} that (\ref{5.21'}) possesses a unique solution $%
(v(\xi ),\pi (\xi ))\in \mathcal{B}_{\mathrm{AP}}^{2}(\mathbb{R}_{\tau };%
\mathbb{B}_{\mathrm{AP}}^{1,2}(G_{2}))\times \mathcal{B}_{\mathrm{AP}}^{2}(%
\mathbb{R}_{\tau };\mathcal{B}_{\mathrm{AP}}^{2}(\mathbb{R}^{N})/\mathbb{R})$%
. This means that $v(\xi )$ is unique up to an additive function $\zeta \in 
\mathcal{B}_{\mathrm{AP}}^{2}(\mathbb{R}_{\tau };\mathbb{B}_{\mathrm{AP}%
}^{1,2}(G_{2}))$ such that $\overline{\nabla }_{y}\zeta =0$ in $G_{2}$ and $%
\pi (\xi )$ is unique up to an additive function which is constant in $G_{2}$%
. Now, fixing $\xi =\frac{\partial }{\partial t}\nabla \boldsymbol{u}%
_{0}(x,t)$ ($(x,t)\in Q$) in (\ref{5.21'}) and repeating the same arguments
as above, we end up with 
\begin{equation*}
\boldsymbol{u}_{1}=v\left( \frac{\partial }{\partial t}\nabla \boldsymbol{u}%
_{0}\right) \text{ in }G_{2}\text{ and }p=\pi \left( \frac{\partial }{%
\partial t}\nabla \boldsymbol{u}_{0}\right) \text{ in }G_{2}
\end{equation*}%
where the functions $v\left( \frac{\partial }{\partial t}\nabla \boldsymbol{u%
}_{0}\right) $ and $\pi \left( \frac{\partial }{\partial t}\nabla 
\boldsymbol{u}_{0}\right) $ are defined in the same manner as $u\left(
\nabla \boldsymbol{u}_{0}\right) $. This being so the functions $\boldsymbol{%
u}_{1}$ and $p$ are therefore defined by 
\begin{equation*}
\boldsymbol{u}_{1}=\chi _{1}u(\nabla \boldsymbol{u}_{0})+\chi _{2}v\left( 
\frac{\partial }{\partial t}\nabla \boldsymbol{u}_{0}\right) \text{ and }%
p=\chi _{2}\pi \left( \frac{\partial }{\partial t}\nabla \boldsymbol{u}%
_{0}\right) .
\end{equation*}

We are now able to derive the effective equations for the equivalent medium.
This is the goal assigned to the next subsection.

\subsection{Homogenization result}

We begin by defining the effective coefficients. Let the matrices $%
C_{0}=(c_{0}^{ij})_{1\leq i,j\leq N}$ and $C_{1}=(c_{1}^{ij})_{1\leq i,j\leq
N}$ be defined as follows: for any $\xi =(\xi _{ij})_{1\leq i,j\leq N}$, 
\begin{eqnarray*}
C_{0}\xi &=&\int_{\mathcal{K}}\widehat{\chi }_{1}(\widehat{A}_{0}(\xi
+\partial \widehat{u(\xi )})+\widehat{A}_{1}\ast \ast (\xi +\partial 
\widehat{u(\xi )}))d\beta \\
C_{1}\xi &=&\int_{\mathcal{K}}\widehat{\chi }_{2}(\widehat{B}_{0}(\xi
+\partial \widehat{v(\xi )})+\widehat{B}_{1}\ast \ast (\xi +\partial 
\widehat{v(\xi )}))d\beta ,
\end{eqnarray*}%
and the function 
\begin{equation*}
h(\xi )=\int_{\mathcal{K}}\widehat{\chi }_{2}\widehat{\pi }(\xi )d\beta .
\end{equation*}%
Then by virtue of the uniqueness of $u(\xi )$, $v(\xi )$ and $\pi (\xi )$
(for a given $\xi $), $C_{0}$, $C_{1}$ and $h$ are well defined. Now,
substituting $\boldsymbol{u}_{1}$ by $\chi _{1}u(\nabla \boldsymbol{u}%
_{0})+\chi _{2}v\left( \frac{\partial }{\partial t}\nabla \boldsymbol{u}%
_{0}\right) $ and $p$ by $\chi _{2}\pi \left( \frac{\partial }{\partial t}%
\nabla \boldsymbol{u}_{0}\right) $ in (\ref{5.14'}) and taking into account (%
\ref{5.17'}), we get the model equation for our equivalent medium, viz. 
\begin{equation}
\left\{ 
\begin{array}{l}
\rho _{0}\frac{\partial ^{2}\boldsymbol{u}_{0}}{\partial t^{2}}-\Div\left(
C_{0}\nabla \boldsymbol{u}_{0}+C_{1}\nabla \frac{\partial \boldsymbol{u}_{0}%
}{\partial t}\right) +\nabla h\left( \nabla \frac{\partial \boldsymbol{u}_{0}%
}{\partial t}\right) =\boldsymbol{F}\text{ in }Q \\ 
\boldsymbol{u}_{0}=0\text{ on }\partial \Omega \times (0,T) \\ 
\frac{\partial \boldsymbol{u}_{0}}{\partial t}(x,0)=\boldsymbol{u}_{0}(x,0)=0%
\text{ in }\Omega .%
\end{array}%
\right.  \label{5.21}
\end{equation}%
The fact that (\ref{5.21}) possesses a unique solution is an easy
consequence of Lemma \ref{l5.2}.

We are therefore led to the homogenization result.

\begin{theorem}
\label{t5.1}Assume that \emph{(A1)-(A4)} hold. For each $\varepsilon >0$,
let $(\boldsymbol{u}_{\varepsilon },\boldsymbol{v}_{\varepsilon
},p_{\varepsilon })$ be the unique solution to \emph{(\ref{1.1})-(\ref{1.7})}%
. Let $\boldsymbol{u}^{\varepsilon }$ denote the global displacement field
defined by \emph{(\ref{5.3})}. Then the sequence $(\boldsymbol{u}%
^{\varepsilon })_{\varepsilon >0}$ strongly converges in $%
H^{1}(0,T;L^{2}(\Omega )^{N})$ to $\boldsymbol{u}_{0}$ where $\boldsymbol{u}%
_{0}$ is the unique solution to \emph{(\ref{5.21})}.
\end{theorem}

\begin{proof}
We just need to prove the convergence of the whole sequence $(\boldsymbol{u}%
^{\varepsilon })_{\varepsilon >0}$. But this is an obvious consequence of
the uniqueness of the solution to (\ref{5.21}).
\end{proof}

\begin{acknowledgement}
\emph{This work has been partly carried out at the Abdus Salam International
Centre for Theoretical Physics (ICTP) in the framework of the Associate and
Federation Schemes. The support of the ICTP is gratefully acknowledged.}
\end{acknowledgement}

\end{document}